\documentclass[11pt,a4paper,reqno]{article}
\usepackage{float}
\usepackage{authblk}
\usepackage{graphicx}
\usepackage{titlesec}
\titleformat{\section}{\center\normalfont\fontsize{13.5}{10}\bfseries}{\thesection}{0.5em}{}
\titleformat{\subsection}{\normalfont\fontsize{12}{17}\bfseries}{\thesubsection}{0.5em}{}
\usepackage{amsmath}
\allowdisplaybreaks
\usepackage[margin=2.5cm,includehead]{geometry}

\usepackage{fancyhdr}
\newcommand\shorttitle{Deformations of Nearly $\G2$-instantons}
\newcommand\authors{Ragini Singhal}

\usepackage{tikz}
\usetikzlibrary{calc}

\usepackage{comment}
\graphicspath{}
\usepackage[linktocpage=true,colorlinks,citecolor=blue,linkcolor=blue,urlcolor=blue, pagebackref]{hyperref}
\usepackage{amssymb}
\usepackage{cite}
\usepackage{amsmath}
\usepackage{latexsym}
\usepackage{amscd}
\usepackage{amsthm}
\usepackage{mathrsfs}
\usepackage{url}
\usepackage{amsmath,mathtools}
\usepackage{cleveref}
\usepackage{hyperref}
\usepackage[alphabetic,backrefs]{amsrefs}
\usepackage{lipsum} 
\usepackage{mathtools}
\usepackage{tikzpagenodes}
\usepackage{faktor}
\usetikzlibrary{tikzmark}
\usepackage{color}
\usepackage[T1]{fontenc}
\usepackage{slashed}

\newtheorem{thm}{Theorem}[section]

\newtheorem{lemma}[thm]{Lemma}
\newtheorem{prop}[thm]{Proposition}

\theoremstyle{definition}

\newtheorem{rem}[thm]{Remark}

\numberwithin{equation}{section}
\def\R{\mathbb R}
\def\S{\slashed{\mathcal{S}}}
\def\O{\mathbb O}
\def\Z{\mathbb Z}
\def\C{\mathbb C}
\def\P{\mathcal P}
\def\d{\mathrm d}

\def\del{\nabla}
\def\G2{\mathrm{G}_2}
\def\g2{\varphi}

\def\red{\color{red}}

\def \m{\mathfrak{m}}
\def\h{\mathfrak{h}} 
\def\fg{\mathfrak{g}}

\def\id{\textup{id}}
\def\cas{\mathrm{Cas}}

\def\lieg2{\mathfrak{g}_2}

\DeclareMathOperator\vol{vol}

\newcommand\Ric{\mathrm{Ric}}
\DeclareFontFamily{U}{MnSymbolC}{}
\DeclareSymbolFont{MnSyC}{U}{MnSymbolC}{m}{n}
\DeclareFontShape{U}{MnSymbolC}{m}{n}{
    <-6>  MnSymbolC5
   <6-7>  MnSymbolC6
   <7-8>  MnSymbolC7
   <8-9>  MnSymbolC8
   <9-10> MnSymbolC9
  <10-12> MnSymbolC10
  <12->   MnSymbolC12}{}
\DeclareMathSymbol{\intprod}{\mathbin}{MnSyC}{'270}

\newcommand\blfootnote[1]{%
  \begingroup
  \renewcommand\thefootnote{}\footnote{#1}%
  \addtocounter{footnote}{-1}%
  \endgroup
}

\begin{document}

\title{Deformations of $\G2$-instantons on nearly $\G2$ manifolds}

\author{Ragini Singhal}

\date{\today}

\maketitle

\textbf{Abstract.} We study the deformation theory of $\G2$-instantons on nearly $\G2$ manifolds. There is a one-to-one correspondence between nearly parallel $\G2$ structures and real Killing spinors, thus the deformation theory can be formulated in terms of spinors and Dirac operators. We prove that the space of infinitesimal deformations of an instanton is isomorphic to the kernel of an elliptic operator. Using this formulation we prove that abelian instantons are rigid. Then we apply our results to describe the deformation space of the canonical connection on the four normal homogeneous nearly $\G2$ manifolds.  
\tableofcontents{}

\blfootnote{ Department of Pure Mathematics, University of Waterloo, Waterloo, ON N2L3G1, \\  email:\ r4singha@uwaterloo.ca} 
\vspace{1.3cm}

\section{Introduction}

Nearly parallel $\G2$ structures on a 7-manifold $M$ are defined by a so-called positive 3-form $\g2$. Such a $3$-form induces a metric $g$, an orientation and a spin structure on $M$ (see \textsection\ref{prelims}). We denote by $\del^g$ the Levi-Civita connection and its lift on the spinor bundle. The $\G2$-structure $\g2$ is nearly parallel if for some $\tau_0\neq 0$
 \begin{align*}
d\g2 & =\tau_0 *_\g2 \g2,
\end{align*}
or equivalently if there exists a real Killing spinor $\eta$ such that 
\begin{align*}
    \del^g_X\eta=-\frac{\tau_0}{8}X\cdot\eta.
\end{align*}

Nearly $\G2$ manifolds were introduced as manifolds with weak holonomy $\G2$ by Gray in \cite{Gray1971}. Some examples of such manifolds are the round and squashed $7$-spheres, the Aloff--Wallach spaces, and the Berger space $\rm{SO}(5)/\rm{SO}(3)$. The inclusion of the exceptional Lie group $\G2$ as a possible holonomy group for Riemannian manifolds in Berger's list \cite{Berger} led mathematicians to look for examples of manifolds with holonomy $\G2$. In \cite{Wang} Wang established the first correspondence between parallel spinors and integrable geometries. Later the classification of manifolds with real Killing spinors in \cites{Hijazi1,Fried-Kath-JDG,  Friedrich1990,Grunewald,Bar_realkilling} established a link between manifolds with weak holonomy and manifolds with real Killing spinors. These manifolds are Einstein with positive scalar curvature. Except for the round $7$-sphere, the dimension of the space of Killing spinors on a nearly $\G2$ manifold is 1,2 or 3 (see\cite{friedkath}) giving rise to three types: proper, Sasaki--Einstein and 3-Sasakian respectively. The cones over these manifolds have holonomy contained in $\textup{Spin}(7)$ which makes these spaces particularly important in the construction and understanding of manifolds with torsion free $\textup{Spin}(7)$-structures.

The correspondence between nearly parallel $\G2$ structures and Killing spinors has been extensively used to produce many results on nearly $\G2$ manifolds. The infinitesimal deformation space of nearly $\G2$-structures was explicitly described as an eigenspace of a Dirac operator in \cite{deformg2}. In the homogeneous setting, non-trivial  deformations were only found for the Aloff--Wallach space and which in \cite{dwivedi2020deformation} were proved to be obstructed.

The spinorial approach can also be used to study gauge theory on manifolds with weak holonomy. A connection $A$ on $M$ is a {\emph{$\G2$-instanton}} if its curvature $F$ satisfies the algebraic condition \begin{align*}
F\wedge \g2 = *_\g2  F,
\end{align*} or equivalently $F\cdot\eta=0$. In this article we describe the infinitesimal deformation space of instantons on nearly $\G2$ manifolds as the eigenspaces of the Dirac operators associated to the one parameter family of connections with skew-symmetric torsion
\begin{align*}
    \del^t_XY=\del^g_XY+\frac{t}{3}\g2(X,Y,\cdot),
\end{align*}
described in \cite{AgrFried,Agr-Evalues,Agr-Ferr,spinor_description,Agr-Holl}. At $t=-1$, the connection $\del^{-1}$ is the \textit{characteristic connection} which is a $\G2$-instanton. We explicitly describe the infinitesimal deformation space of the characteristic connections for the normal homogeneous nearly $\G2$ manifolds classified in \cite{friedkath}. In \cite{bendef} an analogous description for the infinitesimal deformation space of instantons on nearly K\"ahler 6-manifolds is given. On an oriented manifold with real Killing spinor $\eta$ the volume form $\vol$ defines a Killing spinor $\vol\cdot\eta$. On a nearly K\"ahler $6$-manifold  $\{\eta,\vol\cdot\eta\}$ defines a 2 dimensional space of Killing spinors whereas on a nearly $\G2$ manifold $\eta$ and $\vol\cdot\eta$ are linearly dependent. This prevents us from having a relation like in \cite{bendef}*{Proposition 4(iii)} which makes the computation of the infinitesimal deformation space much more convenient (See \ref{rem:difference}). In fact we show in \textsection\ref{instatonhomogeneous} that such a relation does not exist in the nearly $\G2$ case by explicitly computing the kernel of the elliptic operator for the homogeneous nearly $\G2$ manifolds. In \cite{driscoll2020deformations} the author uses the spinorial approach to describe the deformation space of instantons on asymptotically conical $\G2$ manifolds.

\medskip

 Let $M^7$ be a manifold with a $\G2$ structure $\g2$ and let $\eta$ be the Killing spinor associated to $\g2$. A connection $A$ on $M$ is a {\emph{$\G2$-instanton}} if its curvature $F_A$ satisfies the algebraic condition \begin{align*}
F_A\wedge \g2 = *_\g2  F_A.
\end{align*} The above condition is equivalent to $F_A\cdot\eta=0$ as shown in  \textsection \ref{instdef}. When the $\G2$ structure is parallel (the case when the constant $\tau_0 =0$) these instantons clearly solve the Yang--Mills equation $d_\del^* F =0$. The analogous result was proved in the nearly $\G2$ case by Harland--N\"olle \cite{nolle}. They showed that the instantons on manifolds with real Killing spinors solve the Yang--Mills equation which makes the study of instantons on nearly $\G2$ manifolds important from the point of view of gauge theory in higher dimensions.  However $\G2$-instantons in the parallel case are the minimizers of the Yang--Mills functional which is not necessarily true for the nearly parallel case, as proved by  Ball--Oliveira in \cite{gonball}. The first examples of $\G2$-instantons on parallel $\G2$ manifolds were constructed in \cite{clarke}, \cite{walpuski} and \cite{saearp-walpuski}. In \cite{gonball} the authors proved the existence of nearly $\G2$-instantons on certain Aloff--Wallach spaces and classified invariant $\G2$-instantons on these spaces with gauge group U(1) and SO(3).  Recently, Waldron \cite{waldron} proved that the pullback of the standard instanton on $S^7$ obtained from ASD instantons on the 4-sphere via the quaternionic Hopf fibration lies in a smooth, complete, 15-dimensional family of $\G2$-instantons.

\medskip

In \textsection\ref{prelims} we describe a $1$-parameter family of connections on the spinor bundle $\S$ over nearly $\G2$ manifolds and the associated Dirac operators. In \cite{dwivedi2020deformation} the authors introduced a Dirac type operator and used it to completely describe the cohomology of nearly $\G2$ manifolds and proved the obstructedness of infinitesimal deformations of the nearly $\G2$ structure on the Aloff--Wallach space. We remark that the Dirac type operator introduced there is not associated to any connection in the $1$-parameter family.

In \textsection\ref{instdef} we describe the deformation space of a nearly $\G2$ instanton $A$ as an eigenspace of a Dirac operator associated to $A$ and the characteristic connection (Theorem \ref{kernelelliptic}). Using this description, we show that on a compact nearly $\G2$ manifold the $\G2$-instanton $A$ is rigid if the structure group is abelian (cf. Theorem \ref{abelian}(i)) or if all the eigenvalues of a linear operator $L_A$ are greater than $-\frac{28}{5}$ (Theorem \ref{abelian}(ii)) . The instanton $A$ is also rigid if all the eigenvalues of $L_A$ are less than $6$, as shown in \cite{gonball}*{Proposition 8} where authors used a Weitzenb\"ock formula, while the proof of Theorem \ref{abelian}(ii) uses the Schr\"odinger--Lichnerowicz formula for the family of Dirac operators associated to $\del^t$ and $A$.

\medskip

In \textsection\ref{instatonhomogeneous} we describe the infinitesimal deformation space of the characteristic connection on all the homogeneous nearly $\G2$ manifolds whose nearly $\G2$ metric is normal. By considering the actions of the Lie groups $H$ and $\G2$ on $G/H$ we can view the characteristic connection as an $H$-connection or a $\G2$-connection. We compute its infinitesimal deformation spaces in both of these cases. The results are recorded in Theorem \ref{tabledeformationspace}. The deformations are shown to be genuine  in all cases except that of the Aloff--Wallach space $\frac{\rm{SU}(3)\times\rm{SU}(2)}{\rm{SU}(2)\times\rm{U}(1)}$ . In the latter case the author is currently unaware of any known family of nearly $\G2$-instantons  for which the infinitesimal deformations are the ones found in Theorem \ref{tabledeformationspace}.

\vspace{0.7cm} 
\noindent
\textbf{Acknowledgements.} The author would like to thank her supervisors Benoit Charbonneau and Spiro Karigiannis for the innumerable discussions, immense support and advice during the project. The author would also like to thank Gon\c calo Oliveira for recommending Theorem \ref{abelian}(ii), Uwe Semmelmann for his insights, Simon Salamon for the proof of Theorem \ref{squashed-genuine}, and Ilka Agricola for her suggestions which improved the article vastly. The author is also thankful to Shubham Dwivedi for many useful interactions regarding the project.

\medskip

\section{Preliminaries}\label{prelims}
\subsection{Nearly parallel $\G2$ structures}

Let $M$ be a $7$-dimensional Riemannian manifold equipped with a positive $3$-form  $\g2 \in \Omega^3_+(M) $. The $ 3 $-form $\g2$ induces an orientation and a metric on $M$ and thus a Hodge star operator $*_\g2$ on the space of differential forms (see \cite{Bryantholo}). The $\G2$ structure $\g2$ is called a nearly parallel $\G2$ structure on $M$ if it satisfies the following differential equation for some non-zero $ \tau_0 \in \R$, \begin{align}\label{dphi1}
d\varphi = \tau_0*_\g2 \g2.
\end{align} 
\noindent
We denote the $4$-form $*_\g2 \ \g2$ by $\psi$ in the remainder of this article. The condition $d \g2 = \tau_0 \psi$ implies $\d\psi =0$, thus the nearly parallel $\G2$ structure $\g2$ is  co-closed. 
\medskip

Every manifold with a $\G2$ structure is orientable and spin, and thus admits a spinor bundle $\S$. Let $\del^{LC}$ be the Levi-Civita connection of the induced metric on $M$. A spinor $\eta\in\Gamma(\S)$ is a real Killing spinor if for some non-zero $\delta\in\R$, 
\begin{align}
\label{killspin1} \del_X^{LC} \eta = \delta X \cdot \eta \ \ \ \ \ \ \  \forall \ X \in \Gamma (TM).
\end{align} 
There is a one-to-one correspondence between nearly parallel $\G2$ structures and real Killing spinors on $M$. Given a nearly parallel $\G2$ structure $\g2$ that satisfies \eqref{dphi1} there exists a real Killing spinor $\eta$ that satisfies \eqref{killspin1} with $\delta=-\frac{1}{8}\tau_0$ and vice-versa. Switching $-\frac{\tau_0}{8}$ to $\frac{\tau_0}{8}$ corresponds to changing the orientation of the cone $M\times_{r^2} \R^+$. See \cite{book} and \cite{Bar_realkilling} for more details.

The constant $\tau_0$ can be altered by rescaling the metric and readjusting the orientation. In this article we use $\tau_0=4$. With this choice of $\tau_0$ our nearly $\G2$ structure $\g2$ and Killing spinor $\eta$ satisfies the following equations respectively   \begin{align}
d\g2 &= 4\psi,\nonumber \\
\del_X^{LC} \eta &= -\frac{1}{2} X \cdot \eta  \label{killspin}.
\end{align}

Manifolds with nearly parallel $\G2$ structures have several nice properties which can be found in detail in \cite{book}. In particular they are positive Einstein. Let $g$ be the metric induced by $\g2$, then the Ricci curvature  $\Ric_g =\frac 38 \tau_0^2 g$ and the scalar curvature $\textup{Scal}_g = 7 \Ric_g = \frac{21}{8}\tau_0^2$. A $\G2$ structure on $M$ induces a splitting of the spaces of differential forms on $M$ into irreducible $\G2$ representations.  The space of $2$-forms $\Lambda^2(M)$ decomposes as  
\begin{align*}
\Lambda^2(M)&=\Lambda^2_7(M)\oplus \Lambda^2_{14}(M), 
\end{align*}
where $\Lambda^2_l$ has pointwise dimension $l$. More precisely, we have the following description of the space of forms :

 \begin{align*} 
 \Lambda^2_7(M) &=\{X\lrcorner \g2\mid X\in \Gamma(TM)\} = \{\beta \in \Lambda^2(M)\mid *(\g2\wedge \beta)=-2\beta\} , \\
 \Lambda^2_{14}(M) &=\{\beta \in \Lambda^2(M)\mid \beta \wedge \psi =0 \} = \{\beta\in \Lambda^2(M)\mid *(\g2\wedge \beta)=\beta\}. 
 \end{align*}
Note that we are using the convention of \cite{skflow} which is opposite to that of \cite{joycebook} and \cite{bryantrmks}.

The space $\Lambda^2_{14}$ is isomorphic to the Lie algebra of $\G2$ denoted by $\mathfrak{g}_2$. Since the group $\G2$ preserves the $\G2$ structure $\g2$, it preserves the real Killing spinor $\eta$ induced by $\g2$. The space $\Lambda^2_{14}$ can be equivalently defined as\begin{align}\label{instantoneq3}
  \Lambda^2_{14} &=\{\omega \in \Lambda^2 \ | \ \omega \cdot \eta = 0\}.
 \end{align} We make use of this identification when defining the instanton condition on $M$ in \textsection \ref{instdef}.
 
 \subsection{The spinor bundle }


\medskip
 
For a $7$-dimensional Riemannian manifold $M$ with a nearly parallel $\G2$ structure $\g2$, the spinor bundle $\S$ is a rank-8 real vector bundle over $M$ and is isomorphic to the bundle $\underline{\R}\oplus TM = \Lambda^0\oplus \Lambda^1$. At each point $p\in M$, we can identify the fiber of $\S$ with $\R\oplus T_pM \cong \R\oplus \R^7 \cong \textup{Re}(\O)\oplus \textup{Im}(\O) = \O$. If $\eta$ is the real Killing spinor on $M$ induced by $\g2$ then we have the isomorphism\begin{align*}
\S & = (\Lambda^0TM \cdot \eta) \oplus (\Lambda^1 TM \cdot \eta) \cong \Lambda^0 TM \oplus \Lambda^1 TM.
\end{align*}

\noindent
Under this isomorphism any spinor $s = (f\cdot\eta, \alpha\cdot\eta)\in \S$ can be written as $s = (f,\alpha)\in \Lambda^0\oplus\Lambda^1$.

The $3$-form $\g2$ induces a cross product $\times_\g2$ on vector fields $X,Y\in \Gamma(TM)$. Throughout this article we use $e_i$ to denote both tangent vectors and $1$-forms, identified using the metric. All the computations are done in a local orthonormal frame $\{e_1,\dots,e_7\}$ and any
repeated indices are summed over all possible values. With respect to this local orthonormal frame, we have $(X\times_\g2 Y)_l  =X_iY_j\g2_{ijl}$.  The octonionic product of two octonions $(f_1,X_1)$ and $(f_2,X_2)$ is given by,\begin{align*}
(f_1,X_1)\cdot(f_2,X_2) &= (f_1f_2 - \langle X_1,X_2\rangle , f_1X_2 + f_2X_1 + X_1\times X_2).
\end{align*}

As shown in \cite{spironotes} the Clifford multiplication of a 1-form $Y$ and a spinor $(f,Z)$ is the octonionic product of an imaginary octonion and an octonion and is thus given by
 \begin{align}\label{cliffprod}
Y \cdot (f,Z) &= -(-\langle Y,Z \rangle , fY + Y\times Z).
\end{align}
Note that the product defined above differs from \cite{spironotes} by a negative sign due to our choice of the representation of $Cl_7$ on $\S$ \cite{spinbook}*{Chapter 1.8}.
We define the Clifford multiplication of any $p$-form $\beta=\beta_{i_1\ldots i_p}e_{i_1}\wedge e_2\wedge\cdots \wedge e_{i_p}$with a spinor by,
\begin{align*}
\beta \cdot (f,X) &= \beta_{i_1\ldots i_p}(e_{i_1}\cdot (e_{i_2}\cdot\ldots \cdot( e_{i_p} \cdot (f,X))\ldots)).
\end{align*} 

We record an identity for Clifford algebras for later use and refer the reader to \cite{spinbook}*{Proposition 3.8, Ch1} for the proof.

\begin{prop}\label{cliffidenity} For $\alpha \in \Lambda^p(M)$ \begin{align*}
 \sum_j e_j\cdot \alpha \cdot e_j = (-1)^{p+1}(n-2p)\ \alpha.
\end{align*}
\end{prop}

The Clifford multiplication between a $p$-form $\alpha$ and a $1$-form $v$ can be written as \cite{spinbook}*{Proposition 3.9} 
\begin{align*}
    v\cdot \alpha = v\wedge\alpha -v\lrcorner\alpha.
\end{align*}

The vector bundle $\S$ is a $\G2$-representation and since $\G2$ is the isotropy group of the 3-form $\g2$ the map $\mu\mapsto \g2\cdot\mu$ from the bundle of spinors $\S$ to itself is an isomorphism. The same argument holds for the 4-form $\psi$. The following formulae described in \cite{spinor_description,friedivaric} will prove useful in later computations.

\begin{lemma}\label{evalues} The subbundles of $\S$ isomorphic to $\Lambda^0$ and $\Lambda^1$ are eigenspaces of the operations of Clifford multiplication by $\g2$ and $\psi$. The associated eigenvalues are

\begin{center}
\begin{tabular}{c|cc}
& $\Lambda^0$ & $\Lambda^1$  \\ \hline
$\g2$ & $7$ & $-1$ \\
$\psi$ &$  7$ & $-1$. \\
\end{tabular}

\end{center}
\end{lemma}
\begin{proof}
The bundle $\S$ is a $\G2$-representation. The spaces $\Lambda^0, \Lambda^1$ are its irreducible subrepresentations and thus are eigenspaces of the operators defined by the Clifford multiplication by $\g2, \psi$ respectively. By Schur's Lemma there exist real constants $\lambda_0,\lambda_1,\mu_0,\mu_1$ such that for all $f\in \Lambda^0, \alpha\in \Lambda^1$  \begin{align*}
    \g2\cdot f  &= \lambda_0 f, \quad \g2\cdot\alpha = \lambda_1\alpha,\\
    \psi\cdot f  &= \mu_0 f, \quad \psi\cdot\alpha = \mu_1\alpha.
\end{align*}  Proposition \ref{cliffidenity} then implies $\sum_i e_i\cdot\g2\cdot e_i = \g2$ and $\sum_i e_i\cdot\psi\cdot e_i = \psi$ thus \begin{align*}
\lambda_0f&=\g2\cdot f= \sum_{i=1}^7 e_i\cdot\g2\cdot e_i\cdot f ,\\
\mu_0 f&=\psi \cdot f = \sum_{i=1}^7 e_i\cdot\psi\cdot e_i\cdot f.
\end{align*}Using the fact that $e_i\cdot f \in \Lambda^1$ and summing over $i$ we get \begin{align*} \tag{R1}\lambda_0+7\lambda_1&=0,\label{R1}\\ 
\tag{R2} \mu_0+7\mu_1&=0.\label{R2}
\end{align*}

We find the eigenvalues corresponding to $\Lambda^0$ by explicit calculations and use relations (R1) and (R2) to show the result for $\Lambda^1$. 
 Let $(f,0)\in \Lambda^0$ be a spinor.
 \noindent
In the local orthonormal frame $e_1,\ldots,e_7$, we have $\g2 = \frac{1}{6}\g2_{ijk}e_i\wedge e_j\wedge e_k$, where $\g2_{ijk}$ is skew-symmetric in each pair of indices. Using \eqref{cliffprod} we get that\begin{align*}
\g2\cdot(f,0) &= \frac{1}{6}\g2_{ijk}e_i\cdot(e_j\cdot(e_k\cdot(f,0)))= -\frac{1}{6}\g2_{ijk}e_i\cdot(e_j\cdot(0, fe_k ))\\&= \frac{1}{6}\g2_{ijk}e_i\cdot(-f\delta_{kj} , f\g2_{jkt}e_t)\\&= -\frac{1}{6}\g2_{ijk}(-f\g2_{ijk} , -f\delta_{kj}e_i+f\g2_{jkt}\g2_{itp}e_p). 
 \end{align*} By using the skew-symmetry of $\g2$ and the contraction identities $\g2_{ijk}\g2_{ijl} = 6\delta_{kl}, \g2_{ijk}\g2_{ijk} = 42$ (see \cite{skflow}), we get
\begin{align*}
\g2\cdot(f,0) &= \frac{1}{6}(42f, -6f\delta_{it}\g2_{itp}e_p) = (7f,0).\\
\end{align*} 

 \noindent
Similarly, in above local orthonormal frame, $\psi = \frac{1}{24}\psi_{ijkl}e_i\wedge e_j\wedge e_k\wedge e_l$ and using \eqref{cliffprod} we get
\begin{align*}
\psi\cdot (f,0) & = \frac{1}{24}\psi_{ijkl} e_i\cdot(e_j\cdot(e_k\cdot(e_l\cdot(f,0))))= -\frac{1}{24}\psi_{ijkl} e_i\cdot(e_j\cdot(e_k\cdot(0,fe_l ))) \\
& = \frac{1}{24}\psi_{ijkl} e_i\cdot(e_j\cdot(-f\delta_{kl} , f\g2_{klp}e_p))\\
&=-\frac{1}{24}\psi_{ijkl} e_i\cdot( -f\g2_{klp}\delta_{jp},-f\delta_{kl}e_j  +f\g2_{klp}\g2_{jpt}e_t )\\
&=\frac{1}{24}\psi_{ijkl} (f\delta_{kl}\delta_{ij}-f\g2_{klp}\g2_{jpt}\delta_{it}, f\g2_{klp}\delta_{jp}e_i-f\delta_{kl}\g2_{ijs}e_s + f\g2_{klp}\g2_{jpt}\g2_{itr}e_r).
\end{align*}Here we can use the skew-symmetry of $\psi$, the contraction identity $\psi_{ijkl}\g2_{klp} = -4\g2_{ijp}$ along with the contraction identities of $\g2$ mentioned before to obtain
\begin{align*} 
\psi\cdot (f,0)&= \frac{1}{24}(24\delta_{il}\delta_{il}f,0)=\frac{1}{24}(24.7f,0) = (7f,0).
\end{align*}

Substituting  $\lambda_0=7$ and $\mu_0=7$ in Relations \eqref{R1}, \eqref{R2} respectively proves the desired result.
\end{proof}

\medskip

A common feature between nearly K\"ahler $6$-manifolds and manifolds with nearly parallel $\G2$ structures is the presence of a unique canonical connection $\del^{can}$ with totally skew-symmetric torsion defined below. The Killing spinor $\eta$ is parallel with respect to this connection and thus we have $\textup{Hol}(\del^{can}) \subset \G2$. It was proved by Cleyton--Swann in \cite{Cleyton2004}*{Theorem 6.3} that a $G$-irreducible Riemannian manifold $(M, g)$ with an invariant skew-symmetric non-vanishing intrinsic torsion falls in one of the following categories:
\begin{enumerate}
    \item it is locally isometric to a non-symmetric isotropy irreducible homogeneous space, or,
    \item it is a nearly K\"ahler $6$-manifold, or,
    \item it admits a nearly parallel $\G2$ structure.
    \end{enumerate}
    
For the nearly $\G2$ manifold $(M,\g2)$ we define a $1$-parameter family of connections on $TM$ that include the canonical connection $\del^{can}$. Let $t\in \R$ and let  $\del^t$ be the $1$-parameter family of connections on $TM$ defined for all $X,Y,Z \in \Gamma(TM)$ by
\begin{align}
g(\del^t_X Y, Z) & = g(\del^{LC}_X Y, Z) + \frac{t}{3}\g2(X,Y,Z).  \label{delgt} 
\end{align}

\medskip

\noindent
Let $T^t$ be the torsion $(1,2)$-tensor of $\del^t$. Since the connection $\del^{LC}$ is torsion free \begin{align*}
 g(X, T^t(Y,Z)) &= g(X, \del^t_YZ) - g(X, \del^t_ZY) - g(X,[Y,Z]) \\
&= g(\del^{LC}_YZ,X) + \frac{t}{3}\g2(Y,Z,X) - g(\del^{LC}_ZY,X)-\frac{t}{3}\g2(Z,Y,X)\\ &\quad -g(X,[Y,Z]) \\ 
&= \frac{2t}{3}\g2(X,Y,Z).
\end{align*}Therefore the torsion tensor $T^t$ is given by 
\begin{align}
T^t(X,Y)  =\frac{2t}{3}\g2(X,Y,\cdot)\label{torsion}
\end{align} which is proportional to $\g2$ and is thus totally skew-symmetric.

\medskip

By \cite{spinbook}*{Theorem 4.14} the lift of the connection $\del^t$ on the spinor bundle which is also denoted by $\del^t$ acts on sections $\mu$ of $\S$ as \begin{align}
\del^t_X \mu &= \del^{LC}_X \mu + \frac{t}{6}(i_X \g2)\cdot\mu. \label{delt}
\end{align}

\noindent
The space of real Killing spinors is isomorphic to $\Lambda^0$ thus for a Killing spinor $\eta$ it follows from \eqref{killspin} and Lemma \ref{evalues} that for any vector field $X$ since $X\cdot\g2+\g2\cdot X = -2 \ i_X\g2$,  \begin{align*}
\del^t_X \eta &=\del^0_X\eta +\frac{t}{6}(i_X\g2)\cdot\eta\\
&=-\frac{1}{2}X\cdot \eta - \frac{t}{12}(X\cdot \g2 + \g2 \cdot X) \cdot \eta \\ 
& = -\frac{1}{2}X\cdot \eta - \frac{t}{12}( 7 X\cdot\eta - X\cdot \eta ) \\
& = -\frac{t+1}{2}X\cdot \eta.
\end{align*}
\noindent
Therefore $\eta$ is parallel with respect to the connection $\del^{-1}$. The connection $\del^{-1}$ thus has holonomy group contained in $\G2$ with totally skew-symmetric torsion and is therefore the canonical connection on the nearly $\G2$ manifold $M$ described in \cite{Cleyton2004}.

\medskip

\begin{prop}\label{ricciprop} 
The Ricci tensor $\Ric^t$ of the connection $\del^t$ is given by
\begin{align*}
\Ric^t & = (6-\frac{2t^2}{3})g.
\end{align*} \end{prop}

\begin{proof}
By using the expression of the Ricci tensor for a connection with a totally skew-symmetric torsion from \cite{friedivaric}, we have \begin{align*}
\Ric^t(X,Y) &= \Ric^0(X,Y) - \frac{t}{3}\d^* \g2 (X,Y) - \frac{2t^2}{9}g(i_X \g2, i_Y\g2)
\end{align*} 

The Ricci tensor for the Levi-Civita connection is given by $\Ric^0=6g$.  Since $d\psi=0$, $\g2$ is co-closed  and the second term in the above expression vanishes. The third term can be calculated in a local orthonormal frame $e_1,\ldots ,e_7$ using the contraction identity $\g2_{ijk}\g2_{ijl}=6\delta_{kl}$  as follows\begin{align*}
 g(i_X \g2, i_Y\g2)& = \frac{1}{4} \sum_{i,j,k,\alpha,\beta,\gamma} X_kY_\gamma \g2_{ijk}\g2_{\alpha\beta\gamma}g(e_i\wedge e_j, e_\alpha\wedge e_\beta) \\
& = \frac{1}{4}\sum_{i,j,k,\gamma} X_kY_\gamma (\g2_{ijk}\g2_{ij\gamma}-\g2_{ijk}\g2_{ji\gamma})\\
& = 3\sum_{k,\gamma} X_kY_\gamma \delta_{k\gamma} = 3g(X,Y).
\end{align*}

Summing up all the terms together give the desired identity for $\Ric^t$.
\end{proof}

\section{Deformation theory of instantons}\label{instdef}
Let $\P\to M$ be a principal $K$-bundle. We denote by $\textup{Ad}_\P$ the adjoint bundle associated to $\P$. Let $A$ be a connection $1$-form on $\P$ and $F_A\in \Gamma(\Lambda^2T^* M \otimes \textup{Ad}_\P)$ be the curvature $2$-form associated to $A$ given by  \begin{align*}
F_A&=\d A  + \frac{1}{2}[A\wedge A] .
\end{align*}

\medskip

There are many ways to define the instanton condition on $A$. If $(M,g)$ is equipped with a $G$-structure such that $G\subset \mathrm{O}(n)$, there is a subbundle $\mathfrak{g}(T^* M) \subset \Lambda^2T^* M$ whose fibre is isomorphic to $\mathfrak{g}=\textup{Lie}(G)$. The connection $A$ is an instanton if the 2-form part of $F_A$ belongs to $\mathfrak{g}(T^* M)$. In global terms, $A$ is an instanton if \begin{align*}
F_A \in \Gamma(\mathfrak{g}(T^* M) \otimes \textup{Ad}_\P) \subset \Gamma(\Lambda^2T^* M\otimes \textup{Ad}_\P). 
\end{align*} 
Note that in dimension $7$ if $M$ is equipped with a $\G2$ structure then this condition implies that $A$ is an instanton if the $2$-form part of $F_A \in \mathfrak{g}_2(T^*M)=  \Gamma(\Lambda^2_{14}) $. 

\medskip
 
The second definition of an instanton is a special case of the first when the Lie algebra $\mathfrak{g}$ is simple. Its quadratic Casimir is a $G$-invariant element of $\mathfrak{g}\otimes \mathfrak{g}$ which may be identified with a section of $\Lambda^2\otimes \Lambda^2$ and thus to a $4$-form $Q$ by taking a wedge product. Since this $Q$ is $G$-invariant the operator $u\to *(* Q \wedge u)$ acting on $2$-forms commutes with the action of $G$ and hence by Schur's Lemma the irreducible representations of $G$ in $\Lambda^2$ are eigenspaces of the operator. Then $F_A$ is an instanton if 
\begin{align*}
 *(* Q \wedge F_A) &= \nu F_A. 
\end{align*}
for some $\nu\in\R$. In dimension $7$ it turns out that $Q = \psi$ (see \cite{nolle}) and the above condition is equivalent to $F_A\in \Gamma(\Lambda^2_{14})$ when $\nu=1$.
\medskip

Furthermore if $M$ is a spin manifold, and the spinor bundle admits one or more non-vanishing spinors $\eta$, then $A$ is an instanton if 
\begin{align*}
F_A\cdot \eta &= 0. 
\end{align*}
When $M$ has a $\G2$ structure and $\eta$ is the corresponding spinor, \eqref{instantoneq3} implies that a the above condition is satisfied if and only if $A$ is a $\G2$-instanton.  An interested reader can read further on these definitions and their relations in \cite{nolle}.

\medskip

We remark that for an instanton $A$ on a manifold with a $\G2$ structure $\g2$ all the above definitions are equivalent. They all imply that the curvature $F_A$ associated to $A$ lies in $\Gamma(\Lambda^2_{14})$ and thus satisfies \textit{all} of these equivalent conditions:
\begin{align}\label{instantoneq}
\begin{split}
F_A\cdot\eta &=0,\\
F_A\wedge\g2 &= * F,\\
F_A\wedge \psi &=0,\\
F_A\intprod \g2&=0.
\end{split}
\end{align}
 
From now on in this article we use these instanton conditions interchangeably according to the context without further specification. Note that the above definitions are valid for any general $\G2$ structure and not only for nearly parallel ones.

On a manifold with real Killing spinors it was shown in \cite{nolle} that instantons solve the Yang--Mills equation. In the case of a nearly $\G2$-instanton we can prove this fact by direct computation. For an instanton $A$, \eqref{instantoneq} and the second Bianchi identity imply \begin{align*}
(d^A)^* F_A &= * d^A* F_A \\
&= * d^A (\g2 \wedge F_A)\\
& = 4*( \psi\wedge F_A) = 0.
\end{align*}

\subsection{Infinitesimal deformation of instantons}
Let $M^7$ be a nearly $\G2$ manifold. We are interested in studying the infinitesimal deformations of nearly $\G2$-instantons on $M$. An infinitesimal deformation of a connection $A$ represents an infinitesimal change in $A$ and thus, is a section of  $T^* M\otimes Ad_\P$. If $\epsilon\in\Gamma(T^* M\otimes Ad_\P)$ is an infinitesimal deformation of $A$, the corresponding change in the curvature $F_A$ up to first order is given by $\d^A \epsilon$. A standard gauge fixing condition on this perturbation is given by $(\d^A)^* \epsilon=0$. So in total the pair of equations whose solutions define an infinitesimal deformation of an instanton $A$ is given by \begin{align}
(\d^A\epsilon) \cdot \eta = 0, \ \ (\d^A)^{*} \epsilon = 0. \label{perturbation}
\end{align}
On a nearly $\G2$ manifold we can define a $1$-parameter family of Dirac operators 
\begin{align*}
    D^{t,A}&=D^{A}+ \frac{t}{2}\g2\cdot
\end{align*}

\medskip

The 1-parameter family of connections on the spinor bundle $\S$ defined in \eqref{delt} and the connection $A$ on $\P$ can be used to construct a 1-parameter family of connections on the associated vector bundle $\S\otimes Ad_\P$. We denote by $\del^{t,A}$, the connection induced by $\del^t$ and $A$ for all $t\in\R$ respectively. We denote by $D^{t,A}$ the Dirac operator associated to $\del^{t,A}$. The following proposition associates the solutions to \eqref{perturbation} to a particular  eigenspace of $D^{t,A}$ for each $t$. The proposition was proved in \cite{friedevalue} for $t=0$.
\begin{prop}\label{diracevalue}Let $\epsilon$ be a section of $T^* M\otimes Ad_\P$, and let $D^{t,A}$ be the Dirac operator constructed from the connections $\del^{t,A}$ for $t\in \R$. Then $\epsilon$ solves \eqref{perturbation} if and only if \begin{align}
D^{t,A}(\epsilon \cdot \eta) & = -\frac{t+5}{2}\epsilon\cdot\eta. \label{perturbdirac}
\end{align}
\end{prop}

\begin{proof}
Let $\{e_a, a=1\dots 7\}$ be a local orthonormal frame for $T^* M$. Then \begin{align*}
D^{0,A}(\epsilon \cdot \eta) & = e_a\cdot \del^0_a(\epsilon\cdot\eta) \\
&=(e_a\cdot\del^0_a\epsilon)\cdot\eta+e_a\cdot\epsilon\cdot \del^0_a\eta\\
&= (\d^A\epsilon + (\d^A)^* \epsilon)\cdot \eta + e_a\cdot\epsilon\cdot \del^0_a\eta. 
\end{align*}
Applying Proposition \ref{cliffidenity} to the 1-form part of $\epsilon$ we get $e_a\cdot \epsilon\cdot e_a\cdot \eta = 5\epsilon\cdot\eta$. So if $\eta$ is a real Killing spinor then \eqref{killspin} together with the above identity imply
\begin{align*}
D^{0,A}(\epsilon \cdot \eta) & =(\d^A\epsilon + (\d^A)^* \epsilon)\cdot \eta - \frac{1}{2} e_a\cdot\epsilon\cdot e_a \cdot\eta \\
& = (\d^A\epsilon + (\d^A)^* \epsilon -\frac{5}{2} \epsilon)\cdot \eta. 
\end{align*}
It follows from \eqref{delt} and the identity $\sum_a e_a\cdot i_a\g2 = 3\g2$ that
\begin{align*}
D^{t,A} &= D^{0,A} + \frac{t}{2}\g2\cdot 
\end{align*}

Since $\epsilon\cdot\eta\in\Lambda^1\cdot\eta$, by Lemma \ref{evalues} we have \begin{align*}
D^{t,A}(\epsilon\cdot \eta) &= \Big(\d^A\epsilon +(\d^A)^* \epsilon + \frac{-t-5}{2}\epsilon\Big)\cdot \eta. 
\end{align*}

The equation $D^{t,A}(\epsilon\cdot\eta) = -\frac{t+5}{2}\epsilon \cdot \eta$ is thus equivalent to $(\d^A\epsilon +(\d^A)^*\epsilon)\cdot\eta =0$, which in turn is equivalent to the pair of equations $(\d^A\epsilon)\cdot\eta =0, (\d^A)^* \epsilon =0$ since these two components live in complementary subspaces. 
\end{proof}

\medskip

\noindent
Since  $\eta$ is parallel with respect to $\del^{-1}$ we can view $D^{-1,A}$ as an operator on $\Lambda^1\otimes Ad_P$ defined by $D^{-1,A}(\epsilon\cdot\eta) =( D^{-1,A}\epsilon)\cdot\eta$. The following theorem is an immediate consequence of the above proposition.

\begin{thm}\label{kernelelliptic}
 The space of infinitesimal deformations of a $\G2$-instanton $A$ on a principal bundle $\P$ over a nearly $\G2$ manifold $M$ is isomorphic to the kernel of the operator 
 \begin{align}
 \Big(D^{-1,A}+ 2 \ \mathrm{Id}\Big) \colon \Gamma(\Lambda^1\otimes Ad_\P) \to \Gamma(\Lambda^1\otimes Ad_\P).
 \end{align}   
\end{thm}

\medskip

\begin{rem}\label{Dequiv}
By Proposition \ref{diracevalue}, the  $-\frac{t+5}{2}$ eigenspace of the operator $D^{t,A}$ on $\Lambda^1\cdot \eta \otimes Ad_\P$ is isomorphic to the infinitesimal deformation space of the instanton $A$ for all $t\in\R$ and all these eigensapces are thus isomorphic to each other. In particular \begin{align}\label{diraciso}
    \ker(D^{-1/3,A}+\frac{7}{3} \id)\cong \ker(D^{-1,A}+2\id).
\end{align}
\end{rem}
The deformation space found above can be further analysed as an eigenspace of the square of the Dirac operator. In \cite{AgrFried} the authors obtained a Schr\"odinger--Lichnerowicz type formula relating the square of the Dirac operator with torsion $T$ to the connection with torsion $3T$. Such a rescaling was earlier used in \cite{goette} for $\eta$-invariant homogeneous spaces and in \cite{bismut} for Hermitian manifolds. The proof adapted to our setting is presented to keep the discussion self contained. 
\begin{prop}\label{dsquareprop}Let $EM$ be a vector bundle associated to $\P$ and $\mu\in \Gamma(\S\otimes EM)$. Let $A$ be any connection on $\P$. Then for all $t\in\R$,
\begin{align}
(D^{t/3,A})^2\mu &= (\del^{t,A})^*\del^{t,A}\mu + \frac{1}{4}\mathrm{Scal}_g\mu +  \frac{t}{6}\d \g2 \cdot \mu -\frac{t^2}{18}\|\g2\|^2\mu + F\cdot\mu. \label{Dsquare}
\end{align} 
\end{prop}
\begin{proof}
Let $\{e_1,\ldots,e_7\}$ be an orthonormal frame for the tangent bundle. As before we obtain 
\begin{align*}
D^{t,A}\mu &= (D^{0,A} + \frac{t}{2}\g2\cdot)\mu. 
\end{align*}
Squaring both sides we obtain, \begin{align*}
(D^{t/3,A})^2\mu &= \Big( D^{0,A} + \frac{t}{6}\g2\cdot\Big)^2 \mu \\
&= (D^{0,A})^2\mu + \frac{t}{6}(D^{0,A}(\g2\cdot \mu) + \g2\cdot D^{0,A}\mu ) + \frac{t^2}{36}\g2\cdot\g2\cdot\mu.
\end{align*}The first term of the above expression is given by the Schr\"odinger--Lichnerowicz formula \begin{align*}
\tag{E1}\label{E1}(D^{0,A})^2\mu = (\del^{0,A})^*\del^{0,A}\mu + \frac{1}{4}\mathrm{Scal}_g\mu + F\cdot \mu. 
\end{align*}
The anti-commutator in the second term is given by \begin{align*}
D^{0,A}(\g2\cdot \mu) + \g2\cdot D^{0,A}\mu &=e_a\cdot\del_a^{0,A}(\g2\cdot\mu)+\g2\cdot e_a\cdot\del_a^{0,A}\mu \\&=(e_a\cdot\del_a^{0,A}\g2)\cdot\mu+(e_a\cdot\g2+\g2\cdot e_a)\cdot\del_a^{0,A}\mu\\
\tag{E2}\label{E2}&=\d\g2\cdot \mu + \d^*\g2\cdot\mu-2(e_a\intprod\g2)\cdot\del^{0,A}_a\mu 
\end{align*}
but since $M$ is nearly $\G2$, $\g2$ is coclosed, therefore
\begin{align*}
 D^{0,A}(\g2\cdot \mu) + \g2\cdot D^{0,A}\mu&= \d\g2\cdot \mu-2(e_a\intprod\g2)\cdot\del^{0,A}_a\mu
\end{align*}

For the 3-form $\g2$, $\g2\cdot\g2 = \|\g2\|^2-(e_a\intprod\g2)\wedge (e_a\intprod\g2)$ and $(e_a\intprod \g2)\cdot(e_a\intprod \g2) = -3\|\g2\|^2 + (e_a\intprod\g2)\wedge( e_a\intprod\g2)$ which imply \begin{align*}
\g2\cdot\g2\cdot\mu &= \|\g2\|^2\mu -(e_a\intprod\g2)\wedge (e_a\intprod\g2)\cdot\mu,\\
&= \|\g2\|^2 \mu -((e_a\intprod \g2)\cdot(e_a\intprod \g2)+3\|\g2\|^2)\cdot\mu\\
\tag{E3}\label{E3}&= -2\|\g2\|^2\mu - (e_a\intprod \g2)\cdot(e_a\intprod \g2)\cdot\mu.
\end{align*}  

\medskip

At the center of a normal frame, 
\begin{align*}
(\del^{t,A})^* \del^{t,A}\mu &= -(\del^{0,A}_a + \frac{t}{6}(e_a\intprod \g2))(\del^{0,A}_a + \frac{t}{6}(e_a\intprod \g2))\mu \\
&= -\del^{0,A}_a\del^{0,A}_a\mu- \frac{t}{6}(e_a\intprod \g2)\cdot \del^{0,A}_a\mu - \frac{t}{6}\del^{0,A}_a((e_a\intprod \g2)\cdot\mu) \nonumber\\ &\hspace{0.45cm}  -\frac{t^2}{36} (e_a\intprod \g2)\cdot(e_a\intprod \g2)\cdot \mu  \\
&= (\del^{0,A}_a)^*\del^{0,A}_a\mu-\frac{t}{6}(e_a\intprod \g2)\cdot\del^{0,A}_a\mu -\frac{t}{6}(-d^*\g2\cdot\mu + (e_a\intprod\g2)\cdot\del^{0,A}_a\mu) \\&\hspace{0.45cm} - \frac{t^2}{36}((e_a\intprod \g2)\cdot(e_a\intprod \g2))\cdot \mu. \end{align*}
Again using the fact that $d^*\g2=0$ we get
\begin{align*}
\tag{E4}\label{E4}(\del^{0,A}_a)^*\del^{0,A}_a\mu & =(\del^{t,A})^* \del^{t,A}\mu +\frac{t}{3}(e_a\intprod \g2)\cdot \del^{0,A}_a\mu +\frac{t^2}{36}((e_a\intprod \g2)\cdot(e_a\intprod \g2))\cdot \mu.
\end{align*}

Substituting the three terms in the expression of $(D^{t/3,A})^2\mu$ using  \eqref{E1}, \eqref{E2}, \eqref{E3}, \eqref{E4} we get the result.\end{proof}

When the connection $A$ is an instanton on a nearly $\G2$ manifold the expression for $(D^{t/3,A})^2$ can be simplified further. For the $\G2$ structure $\g2$, $\|\g2\|^2=7$ and under our choice of convention $d\g2=4\psi$ and $\textup{Scal}_g =42$. Thus we can calculate the action of $(D^{t/3,A})^2$ on spinors in $\Lambda^0\eta$ and $\Lambda^1\cdot\eta$ as follows.

Let $\eta\in \Gamma( \Lambda^0M\otimes EM)$ be a real Killing spinor then Lemma \ref{evalues} implies $\psi\cdot\eta=7\eta$  and $F_A\cdot\eta=0$ by \eqref{instantoneq}. Thus by above proposition we obtain,\begin{align}
(D^{t/3,A})^2\eta &= (\del^{t,A})^*\del^{t,A}\eta -\frac{7}{18}(t^2-12t-27)\eta. \label{dsquareeta}
\end{align}

Now suppose $\epsilon$ is an infinitesimal deformation of $A$. Then $\epsilon\cdot\eta\in  \Gamma( \Lambda^1M\otimes  EM)$. From Lemma \ref{evalues} we know that $\psi\cdot\epsilon\cdot\eta =-\epsilon\cdot\eta$ and since $F\cdot\eta=0$, $F\cdot\epsilon\cdot\eta = (F\cdot\epsilon+\epsilon\cdot F)\cdot\eta = -2(\epsilon\lrcorner F) \cdot \eta $. Thus by above proposition \begin{align} 
(D^{t/3,A})^2(\epsilon\cdot\eta) &= (\del^{t,A})^*\del^{t,A}(\epsilon\cdot\eta)-\frac{1}{18}(7t^2+12t-189)\epsilon\cdot\eta -2(\epsilon\lrcorner F) \cdot \eta. \label{dsquaremu}
\end{align} 
\noindent

In the special case when the bundle $EM$ is equal to $Ad_\P$, the holonomy group $H$ $\subset$ $G$ of the connection $A$ acts on the Lie algebra $\mathfrak{g}$ of $G$.  Let us denote by $\mathfrak{g}_0 \subset \mathfrak{g}$ the subspace on which $H$ acts trivially. Let $\mathfrak{g}_1$ be the orthogonal subspace of $\mathfrak{g}_0$ with respect to the Killing form of $G$.  The corresponding splitting of the adjoint bundle is given by $
Ad_\P = L_0\oplus L_1.$ By Proposition \ref{dsquareprop} $(D^{-1/3,A})^2$ is self adjoint and hence respects the decomposition \begin{align*}
\S\otimes Ad_\P &= (\Lambda^1M\otimes L_0)\oplus (\Lambda^1M\otimes L_1)\oplus (\Lambda^0M\otimes L_0)\oplus (\Lambda^0M\otimes L_1).
\end{align*}
We use the shorthand $\Lambda^iL_j$ for $\Lambda^iM\otimes L_j$ where $i,j=0,1$. For compact $M$ we have the following proposition.

\begin{prop}\label{kerprop}
Let $A$ be a $\G2$-instanton on a principal $G$-bundle $\P$ with holonomy group $H$ and suppose $Ad_\P$ splits as above. Then \begin{enumerate}
\item[\textup{(i)}] $\ker((D^{-1/3,A})^2-\frac{49}{9}\id) = \ker((D^{-1/3,A})^2-\frac{49}{9}\id)\cap (\Lambda^1L_1\oplus \Lambda^0L_0).$
\item[\textup{(ii)}]$\ker((D^{-1/3,A})^2-\frac{49}{9}\id)\cap \Lambda^1 L_1 = \Big(\ker(D^{-1/3,A}+ \frac{7}{3}\id)\oplus  \ker(D^{-1/3,A}- \frac{7}{3}\id) \Big)\cap \Lambda^1L_1$.
\end{enumerate}
\end{prop}
\begin{proof}

To prove (i) we need to show that $\ker((D^{-1/3,A})^2-(\frac{7}{3} )^2\id)\cap( \Lambda^0L_1 \oplus \Lambda^1L_0)$ is trivial.

\begin{enumerate}
\item Let $\mu\in\ker((D^{-1/3,A})^2-(\frac{7}{3} )^2\id)\cap \Lambda^0L_1$. Thus we have by \eqref{dsquareeta} ,\begin{align*}
0&=\int_{M} (\mu,(D^{-1/3,A})^2-(\frac{7}{3} )^2)\mu)  \\
 &=\int_M(\mu, (\del^{-1,A})^*\del^{-1,A}\mu +(\frac{49}{9}-\Big(\frac{7}{3}\Big)^2)\mu)\\
 &=\int_M \|\del^{-1,A}\mu\|^2.
\end{align*} But since the action of the holonomy group of $A$ fixes no non-trivial elements in $\mathfrak{g}_1$ and the holonomy group of $\del^{-1}$ acts trivially on $\Lambda^0$ we get $\mu=0$.

\item  Let $\epsilon\cdot\eta\in\ker((D^{-1/3,A})^2-(\frac{7}{3} )^2\id)\cap \Lambda^1L_0$. By the definition of $L_0$ the curvature $F_A$ acts trivially on $\epsilon\cdot\eta$  in \eqref{dsquaremu} and we get,\begin{align*}
0&=\int_{M} (\epsilon\cdot\eta,(D^{-1/3,A})^2-(\frac{7}{3} )^2)\epsilon\cdot\eta) \\
&=\int_M(\epsilon\cdot\eta, (\del^{-1})^*\del^{-1}(\epsilon\cdot\eta) +(\frac{97}{9}-\Big(\frac{7}{3}\Big)^2)\epsilon\cdot\eta)\\
&=\int_M \|\del^{-1}(\epsilon\cdot\eta)\|^2 + \frac{48}{9}\int_M\|\epsilon\cdot\eta\|^2
\end{align*}
hence $\epsilon\cdot\eta=0$.

\end{enumerate}

For proving (ii) we already know that $\left( \ker((D^{-1/3,A})+ \frac{7}{3}\}\oplus \ker((D^{-1/3,A})+ \frac{7}{3}\} \right)\cap\Lambda^1L_1\subset \ker((D^{-1/3,A})^2- \frac{49}{9}\id)\cap\Lambda^1L_1$. The reverse inclusion can be seen using the fact that since $D^{-1/3,A}$ and $(D^{-1/3,A})^2$ commute they have the same eigenvectors. Moreover since $D^{-1/3,A}$ is self adjoint, $\epsilon\cdot\mu \in \ker((D^{-1/3,A})^2- \frac{49}{9}\id)\cap\Lambda^1L_1$ implies $\|D^{-1/3,A}\epsilon\cdot\mu \| =  \frac{7}{3}  \|\epsilon\cdot\mu\|$ thus the corresponding eigenvalues of $D^{-1/3,A}$ can only be $\pm\frac{7}{3}$ .
\end{proof}

\begin{rem}
Note that part (i) for the above proposition holds only for $D^{-1/3,A}$ and not for any other $D^{t,A}$ where $t\neq -1/3$ since the proof explicitly uses the fact that $\eta$ is parallel with respect to $\del^{-1}$. But since $D^{t,A}$ is self adjoint for all $t\in\R$, for any $\lambda\in \R$ we have the following decomposition 
\begin{align*}
    \ker\left\{ (D^{t,A})^2-\lambda^2\id\right\}\cap \Lambda^1Ad_\P=\left(\ker\left\{D^{t,A}-\lambda\id\right\}\oplus\ker\left\{D^{t,A}+\lambda\id\right\}\right)\cap \Lambda^1Ad_\P.
\end{align*}
\end{rem}
The above proposition has the following important consequence. If the structure group $G$ is abelian $H$ acts as identity on the whole of $\mathfrak{g}$ which means $\mathfrak{g}_1 =0$ and  $L_1$ is trivial. Thus by Remark \ref{Dequiv} the space of infinitesimal deformations of the $\G2$-instanton $A$ which is isomorphic to $\ker(D^{-1/3,A} + \frac{7}{3} )\cap \Lambda^1 Ad_\P =\ker(D^{-1/3,A} + \frac{7}{3} )\cap \Lambda^1 L_1 $ is zero dimensional.
  
In  \cite{gonball}*{Proposition~24} the authors prove that the $\G2$-instanton $A$ is rigid if all the eigenvalues of the operator 
\begin{align*}
     L_A\colon \Lambda^1\otimes Ad_\P &\to\Lambda^1\otimes Ad_\P \nonumber\\w&\mapsto -2w\lrcorner F_A
\end{align*} 
are smaller than $6$. We prove the lower bound for the eigenvalue as follows. Let $\lambda$ be the smallest eigenvalue of $L_A$. If $\epsilon\in \Gamma(T^*M\otimes Ad_\P)$ is an infinitesimal deformation of $A$ then from \eqref{dsquaremu} and Theorem \ref{kernelelliptic} we know that
\begin{align*}
    (\del^{t,A})^*\del^{t,A}\epsilon\cdot\eta &= \Big(\frac{5t^2}{12}+\frac{3t}{2}-\frac{17}{4}\Big)\epsilon\cdot\eta-L_A(\epsilon)\cdot\eta .
\end{align*} Taking the inner product with $\epsilon\cdot\eta$ on both sides we get that if  $\lambda >\min\left\{\frac{5t^2+18t-51}{12} \ | \ t\in\R\right\} = -\frac{28}{5}$ then $\epsilon=0$ is the only solution. Thus we get the following result.
\begin{thm}\label{abelian}Any $\G2$-instanton $A$  on a  principal $G$-bundle over a compact nearly $\G2$ manifold M is rigid if \begin{itemize}
    \item[\textup{(i)}]the structure group $G$ is abelian, or
    \item[\textup{(ii)}] the eigenvalues of the operator $L_A$ are either all greater than $-\frac{28}{5}$ or all smaller than $6$.
\end{itemize} 
\end{thm}

Some immediate consequences of Theorem \ref{abelian} are that the flat instantons are rigid. Also if all the eigenvalues of $L_A$ are equal then $A$ has to be rigid.
\medskip


\section{Instantons on homogeneous nearly $\G2$ manifolds }\label{instatonhomogeneous}

\subsection{Classification of homogeneous nearly $\G2$ manifolds}
By the classification result in \cite{friedkath} there are six compact, simply connected homogeneous nearly $\G2$ manifolds:

\begin{table}[H]
 \centering
    \def\arraystretch{2.5}
\begin{tabular}{c c c}
\centering
$(S^7, g_{round})= \mathrm{Spin(7)}/\G2$,&$(S^7$, $g_{squashed})= \frac{\mathrm{Sp}(2)\times\mathrm{Sp}(1)}{\mathrm{Sp}(1)\times \mathrm{Sp}(1)}$,& $\mathrm{SO(5)/SO(3)}$,\\
$M(3,2)= \frac{\mathrm{SU}(3)\times\mathrm{SU}(2)}{\mathrm{U}(1)\times \mathrm{SU}(2)}$,&$N(k,l)= \mathrm{SU}(3)/S^1_{k,l} \ k,l\in\Z $, &$Q(1,1,1) = \mathrm{SU(2)^3}/\mathrm{U(1)^2}$.
\end{tabular}
\end{table}
where $S^1_{k,l}=\{\textup{diag}(e^{ik\theta},e^{il\theta},e^{-i(k+l)\theta}),\theta\in\R\}$ denotes the embedding of $\mathrm{U}(1)$ into $\mathrm{SU}(3)$. We describe the homogeneous structure on each of these spaces.

\noindent
\begin{itemize}
    \item[-] In the round $S^7$ the embedding of $\G2$ in $\mathrm{Spin}(7)$ is obtained by lifting the standard embedding of $\G2$ into $\mathrm{SO}(7)$.
    \item[-] For the squashed metric on $S^7$ the two copies of $\mathrm{Sp}(1)$ in $\mathrm{Sp}(2)\times\mathrm{Sp}(1)$ denoted by $\mathrm{Sp}(1)_u$ and $\mathrm{Sp}(1)_d$ \cite{deformg2} are 
\begin{align*}
    \mathrm{Sp}(1)_u&:= \left\{\left(\begin{pmatrix}a&0\\0&1\end{pmatrix} ,1 \right)\colon a\in\mathrm{Sp}(1)\right\},\ \ \  \mathrm{Sp}(1)_d:= \left\{\left(\begin{pmatrix}1&0\\0&a\end{pmatrix} ,a \right)\colon a\in\mathrm{Sp}(1)\right\}.
\end{align*}

\item[-] In the Berger space $\frac{\mathrm{SO}(5)}{\mathrm{SO}(3)}$, the Lie group $\mathrm{SO}(3)$ is embedded into $\mathrm{SO}(5)$ via the $5$ dimensional irreducible representation of $\mathrm{SO}(3)$ on $\mathrm{Sym}^2_0(\R^3)$.
\item[-] In $\frac{\mathrm{SU}(3)\times\mathrm{SU}(2)}{\mathrm{U}(1)\times \mathrm{SU}(2)}$ the embedding of $\mathrm{SU}(2)$ (denoted by $\mathrm{SU}(2)_d$) and $\mathrm{U}(1)$ in $\mathrm{SU}(2)\times\mathrm{SU}(2)$ is defined as \cite{deformg2}
\begin{align*}
\mathrm{SU}(2)_d&:= \left\{\left(\begin{pmatrix}a&0\\0&1\end{pmatrix} ,a \right)\colon a\in\mathrm{SU}(2)\right\},\ \ \ \mathrm{U}(1):=  \left\{\left(\begin{pmatrix}e^{i\theta}&0&0\\0&e^{i\theta}&0\\0&0&e^{-2i\theta}\end{pmatrix} ,1 \right)\colon \theta\in\R\right\} 
\end{align*}
\item[-] In the Aloff-Wallach spaces $N_{k,l}$ where $k,l$ are coprime positive integers the embedding of $S^1_{k,l}=U(1)_{k,l}$ in $\mathrm{SU}(3)$ is described
\begin{align*}
    S^1_{k,l}&=\left\{\begin{pmatrix}e^{ik\theta}&0&0\\0&e^{il\theta}&0\\0&0&e^{-i(k+l)\theta}\end{pmatrix},\theta\in\R\right\}
\end{align*}
\item[-] In $Q(1,1,1)$ we denote the two copies of $\mathrm{U}(1)$ inside $\mathrm{SU}(2)^3$ as $\mathrm{U}(1)_u,\mathrm{U}(1)_d$ where their respective embeddings are given by  \begin{align*}
\mathrm{U}(1)_u&= \textup{Span}\left\{\left(\begin{pmatrix}
e^{i\theta}&0\\0&e^{-i\theta}
\end{pmatrix},\begin{pmatrix}
e^{-i\theta}&0\\0&e^{i\theta}
\end{pmatrix},\textup{I}_{2}\right),\theta\in\R\right\}, \\ 
\mathrm{U}(1)_d&= \text{Span}\left\{\left(\textup{I}_{2},\begin{pmatrix}
e^{i\theta}&0\\0&e^{-i\theta}
\end{pmatrix},\begin{pmatrix}
e^{-i\theta}&0\\0&e^{i\theta}
\end{pmatrix}\right),\theta\in\R\right\}.
\end{align*}

\end{itemize}

The first four homogeneous spaces are normal, and for those the nearly $\G2$ metric $g$ on $G/H$ is related to the Killing form $B$ of $G$ by $g=-\frac{3}{40}B$. The choice of the scalar constant $\frac{3}{40}$ is based on our convention $\tau_0=4$. The general formula for the constant was derived in \cite{deformg2}*{Lemma~7.1}. In the remaining two homogeneous spaces the nearly $\G2$ metric is not a scalar multiple of the Killing form of $G$ (see \cite{wilkingnkl}).
\begin{table}[H]
\begin{center}
\def\arraystretch{2.5}
\begin{tabular}{c c}
\centering
$(S^7, g_{round})\cong \mathrm{Spin(7)}/\G2$,&$(S^7$, $g_{squashed})\cong \frac{\mathrm{Sp}(2)\times\mathrm{Sp}(1)}{\mathrm{Sp}(1)\times \mathrm{Sp}(1)}$,\\ $\mathrm{SO(5)/SO(3)}$,
&$M(3,2)\cong \frac{\mathrm{SU}(3)\times\mathrm{SU}(2)}{\mathrm{U}(1)\times \mathrm{SU}(2)}$.
\end{tabular} 
\end{center}
\caption{Normal homogeneous nearly $\G2$ manifolds}
\label{tablenormal}
\end{table}

Let $\m$ be the orthogonal complement of the Lie algebra $\h$ of $H$ in $\fg$ with respect to $g$. Then $\m$ is invariant under the adjoint action of $\h$ that is, $[\h,\m]\subset \m$ and thus all the six homogeneous nearly $\G2$ manifolds are naturally reductive. The reductive decomposition $\fg = \h \oplus\m$ equips the principal $H$-bundle $G \to  G/H$ with a $G$-invariant connection
whose horizontal spaces are the left translates of $\m$. This connection is known as the characteristic homogeneous connection. On homogeneous nearly $\G2$ manifolds the characteristic homogeneous connection has holonomy contained in $\G2$. If we denote by $Z_\m$ the projection of $Z\in\fg$ on $\m$, the torsion tensor $T$ for any $X,Y\in\m$ is given by 
\begin{align*}
    T(X,Y)&=-[X,Y]_\m,
\end{align*} and is totally skew-symmetric. Thus by the uniqueness result in \cite{Cleyton2004} it
is the canonical connection with respect to the nearly $\G2$ structure on $G/H$ \cite{nolle}. The canonical connection is a $\G2$-instanton as proved in \cite{nolle}*{Proposition~3.1}.

The adjoint representation $\operatorname{ad}\colon H\to \mathrm{GL}(\m)$  gives rise to the associated vector bundle $G\times_{\operatorname{ad}}\m$ on $G/H$. 
Similarly since $G/H$ has a nearly $\G2$ structure we have the adjoint action of $\G2$ on $\m$ which we again denote by $\operatorname{ad}$ and the isotropy homomorphism $\lambda\colon H\to \G2$ which we can use to construct the associated vector bundle $G\times_{\operatorname{ad}\circ\lambda}\m$. 
The canonical connection is a connection on both $G\times_{\operatorname{ad}}\m$ and $G\times_{\operatorname{ad}\circ\lambda}\m$ with structure group $H$ and $\G2$ respectively. Therefore it is natural to study the infinitesimal deformation space of the canonical connection in both these situations. Since $H\subset \G2$, the deformation space as an $H$-connection is a subset of the deformation space as a $\G2$-connection. 

We can completely describe the deformation space when the structure group is $H$ but for structure group $\G2$ we can only find the deformation space for the normal homogeneous nearly $\G2$ manifolds listed in Table \ref{tablenormal} since our methods do not work for non-normal homogeneous metrics. However since $H$ is abelian in both of the non-normal cases Theorem \ref{abelian} tells us that the canonical connection is rigid as an $H$-connection. But we cannot say anything about the deformation space for the structure group $\G2$ in those two cases.

Thus the only cases left to consider are listed in Table \ref{tablenormal}. The  remainder of this article is devoted to computing the infinitesimal deformation space of the canonical connection with the structure group $H$ and $\G2$ for the homogeneous spaces listed in Table \ref{tablenormal}.

\subsection{Infinitesimal deformations of the canonical connection }

Let $M=G/H$ be a homogeneous manifold. Consider the principal $H$-bundle $G\to M$. If $(V,\rho)$ is an $H$-representation then the space of smooth sections $\Gamma(G\times_\rho V)$ of the associated vector bundle $G\times_\rho V$  is isomorphic to the space $C^\infty(G,V)_H$ of $H$-equivariant smooth functions $G\to V$. The space $C^\infty(G,V)_H$ carries the left regular $G$-representation  $\rho_L$ defined by $\rho_L(g)(f)=g.f=f\circ l_{g^{-1}}$ which is also known as the induced $G$-representation $\textup{Ind}_{H}^G V$. 

\noindent
For any connection $A$ on $G$ the covariant derivative associated to $A$ on any bundle associated to $A$ is denoted by $\del^{A}$. Let $s\in \Gamma(G\times_\rho V)$ and $f_s \colon G\to V$ be the $G$-equivariant function given by $s(gH)=[g,f_s(g)]$. If we denote by $X_h$ the horizontal lift of $X\in\Gamma(TM)$ via $A$, then $\del^A$ acts on  $s$ as \begin{align*}
(\del^{A}_X s)(gH) &= (g,X_h(f_s)(g)).
\end{align*}

\noindent
For the canonical connection on $G\to M$, $X_h=X$ for every vector field. Thus the covariant derivative $\del^{can}$ is given by \begin{align*}
(\del^{can}_X s)(gH)&= (g, X(f_s)(g)). 
\end{align*} 

By the Peter--Weyl Theorem \cite{knappbook}*{Theorem 1.12} the space of sections can also be formulated as follows. If we denote by $G_{irr}$ the set of equivalence classes of irreducible $H$-representations then 
\begin{align*}
    \Gamma(G\times_\rho V) &=\overline{\underset{W\in{G_{irr}}}{\bigoplus}\mathrm{Hom}(W,V)_{H}\otimes W}.
\end{align*} 
 The embedding $\textup{Hom}(W,V)_{H}\otimes W$ into $C^\infty (G,V)_H=\Gamma(G\times_\rho V)$ is given by sending $(\phi,w)$ to the function $f_{(\phi,w)}$ defined by $f_{(\phi,w)}(g)= \phi(\tau(g^{-1})w)$. Thus $(\phi,w)$ defines a section $s_{(\phi,w)}(gH)=[g,f_{(\phi,w)}(g)]$ which we denote by $(\phi,w)$ as well.
\medskip

\noindent
\textbf{Claim:} The left $G$-action is given by $g.f_{(\phi,w)}=f_{(\phi,\tau(g)w)}$.

\noindent
\begin{proof}
Let $k\in G$. Then since $g.f=f\circ l_{g^{-1}}$ and $f_{(\phi,w)}(g)=\phi(\tau(g^{-1})w)$ we have   
\begin{align*}
    (g.f_{(\phi,w)})(k) &=f_{(\phi,w)}(g^{-1}k)= \phi(\tau((g^{-1}k)^{-1})w) \\
    &= \phi(\tau(k^{-1})\tau(g)w)=f_{(\phi,\tau(g)w)}(k).
\end{align*}
The proof of the claim is now complete.\end{proof}  

We can compute the covariant derivative on $s_{(\phi,w)}\in \textup{Hom}(W,V)_H\otimes W \subset \Gamma(G\times_\rho V)$ by
\begin{align*}
    \del^{can}_X s_{(\phi,w)}(gH)&=X(f_{(\phi,w)})(g)=\frac{d}{dt}\Big|_{t=0} f(e^{tX}g)\\
    &=\frac{d}{dt}\Big|_{t=0}(f_{(\phi,w)}\circ l_{e^{tX}})(g)=\frac{d}{dt}\Big|_{t=0}(e^{-tX}. f) (g)\\
    &=\frac{d}{dt}\Big|_{t=0}f_{(\phi,\tau(e^{-tX})w)}=-f_{(\phi,\tau_*(X)w)}(gH).
\end{align*}
The above can be written as 
\begin{align}\label{covariantcan}
    \del^{can}_X(\phi,w)=-(\phi,\tau_*(X)w).
\end{align}
Thus we get that for the canonical connection the covariant derivative of a section $s\in\Gamma(G\times_\rho V)$ with respect to
some $X \in\m$ translates into the derivative $X(f_s)$, which is minus the differential of the left-regular representation $(\rho_L)_* (X)(f_s)$, see \cite{HYMLaplace}. 
\medskip

Let 
$\{a_i,\ i=1\ldots n\}$ be an orthonormal basis of $\mathfrak{\fg}$ with respect to $g=-\frac{3}{40}B$ then the Casimir element $\cas_\mathfrak{\fg}\in \textup{Sym}^2(\mathfrak{\fg})$ is defined by $\sum_{i=1}^{\dim G}a_i\otimes a_i$. On any $\mathfrak{\fg}$ representation $(V,\mu)$ we can define the Casimir invariant $\mu(\cas_\mathfrak{\fg})\in \mathrm{gl}(V)$ by \begin{align*}
    \mu(\cas_\mathfrak{\fg})&=\sum_{i=1}^{n} \mu(a_i)^2.
\end{align*}
For the reductive homogeneous spaces $G/H$ let $\{a_i,i=1\ldots\dim(H)\}$ and $\{a_i,i=\dim(H)\ldots \dim(G)\}$ be the basis of $\h$ and $\m$ respectively. If we  define $\cas_\h=\sum_{i=1}^{\dim(H)}a_i\otimes a_i$ and $\cas_\m=\sum_{\dim(H)}^{\dim(G)}a_i\otimes a_i$ we can decompose $\cas_\fg$ as 
\begin{align*}
    \cas_\fg&=\cas_\h+\cas_\m.
\end{align*} 
 Note that $\cas_\m$ is just used for notational convenience and as $\m$ may not be a Lie algebra apriori. Also in $\cas_\h$ the trace is taken over $H$.

\begin{rem}
If one uses the metric $-cB$ instead of $-B$ then the Casimir operator is divided by the scalar $c$.
\end{rem}

\noindent
To study the deformation space of the canonical connection $\del^{can}$ on these homogeneous spaces we rewrite the Schr\"odinger--Lichnerowicz formula \eqref{dsquaremu} in terms of the Casimir operator of $\h$ and $\fg$ and then use the Frobenius reciprocity formula to compute the deformation space of the canonical connection in each case. Let $F$ be the curvature associated to $\del^{can}$ then the operator $-2\epsilon\lrcorner F$ can be reformulated in terms of $\cas_\h$ by doing similar calculations as in \cite{bendef}*{Lemma~4} which gives \begin{equation}\label{Fcas}
-2\epsilon\lrcorner F = (\rho_{\m^*}(\cas_\h)\otimes 1_E+1_{\m^*}\otimes \rho_E(\cas_\h)-\rho_{\m^*\otimes E}(\cas_\h))\epsilon.
\end{equation} 

Let $(E,\rho_E)$ be an $H$-representation. We denote the tensor product of representations on $\m^*$ and $E$ by $\rho_{m^*\otimes E}$. For every $t\in\R$, $D^{t,A}$ denotes the Dirac operator  on $G\times_{\rho_{\m^*\otimes E}}(\m^*\otimes E)\otimes\S$ associated to the connection $\del^{A}$ and $\del^t$ on $G\times_{\rho_{\m^*\otimes E}}(\m^*\otimes E)$ and $\S$ respectively. From now on we use the same symbol to denote the Lie group representation and the associated Lie algebra representation wherever there is no confusion.  On a naturally reductive space Kostant's formula for cubic Dirac operator relates the square of the Dirac operator to suitable Casimir operators and scalar terms (see \cite{kostant,Agr-cubicoperator, HYMLaplace}).
We now use Proposition \ref{dsquareprop} to prove a similar result for $(D^{-1/3,c})^2$. 

\begin{prop}\label{casprop}Let $\del^{can}$ be the canonical connection on a homogeneous nearly $\G2$ manifold $M=G/H$. Let $(E,\rho_E)$ be an $H$-representation  and $\epsilon$ be a smooth section of $G\times_{\rho_{\m^*\otimes E}}(\m^*\otimes E)$. Then \begin{equation}
(D^{-1/3,can})^2\epsilon\cdot\eta = (-\rho_L(\cas_\fg) + \rho_E(\cas_\h))\epsilon + \frac{49}{9}\epsilon)\cdot\eta.
\end{equation} 
\end{prop}
\begin{proof}
We begin by analyzing the rough Laplacian term in the Schr\"odinger-- Lichnerowicz  formula for $(D^{-1/3,can})^2\epsilon\cdot\eta$ from \eqref{dsquaremu} and then substitute the $F$-dependent term from \eqref{Fcas} in the same. We denote by $\rho_L$ the left regular representation of $G$.
From above calculations we know that at the center of a normal orthonormal frame $\{e_i,i=1\ldots 7\}$ of $\m$ with respect to $g=-\frac{3}{40}B$,
\begin{align*}
(\del^{-1,can})^* \del^{-1,can} &= -\del_{e_i}^{-1,can}\del_{e_i}^{-1,can} = -\rho_L(e_i)^2= -\rho_L(\cas_\m).
\end{align*} 
Since $\textup{Res}_{G}^{H}\ \rho_L = \textup{Res}_{G}^{H}\ \textup{Ind}_{H}^G (m^*\otimes E)\cong  m^*\otimes E$ we have that    $\rho_{\m^*\otimes E}(\cas_\h)=\rho_L(\cas_\h)$.  Also $\rho_{\m^*}(e_i)^2=\rho_{\m^*}(\cas_\h)$ acts as $-\Ric$ of the canonical connection on $1$-forms which is equal to $-\frac{16}{3} \id$ from Proposition \ref{ricciprop}. Substituting all the terms in \eqref{dsquaremu} for $t=-1$ we get \begin{align*}
(D^{-1/3,can})^2\epsilon\cdot\eta &= (-\rho_L(\cas_\m)\epsilon+\frac{97}{9}\epsilon + (\rho_{\m^*}(\cas_\h)\otimes 1_E+1_{\m^*}\otimes \rho_E(\cas_\h)-\rho_{\m^*\otimes E}(\cas_\h))\epsilon)\cdot\eta\\
&=(-(\rho_L(\cas_\m)+\rho_L(\cas_\h))\epsilon +(\frac{97}{9}-\frac{16}{3})\epsilon + \rho_E(\cas_\h)\epsilon )\cdot\eta\\
&=( (-\rho_L(\cas_\fg)\epsilon + \rho_E(\cas_\h)\epsilon + \frac{49}{9}\epsilon)\cdot\eta
\end{align*}which completes the proof.\end{proof}

Since all the homogeneous spaces considered in Table \ref{tablenormal} are naturally reductive and $H\subset \G2$, there is an adjoint action of $H$ on $\m,\h$ and $\lieg2$ and thus $H$-representations on $\m^*\otimes \h$ and $\m^*\otimes\fg$ which we denote by $\rho_{\m^*\otimes \h},\rho_{\m^*\otimes \lieg2}$. The corresponding Lie algebra representations are denoted similarly. The infinitesimal deformation space of the instanton $\del^{can}$ is a subspace of $\Gamma(\m^*\otimes E)$ where $E$ can be either $\h$ or $\lieg2$. 

From Propositions \ref{diracevalue} and \ref{casprop} it is clear that if $\epsilon$ is an infinitesimal deformation of $\del^{can}$ on the bundle $\m^*\otimes E$ over  $G/H$ then 
\begin{align}\label{caseq}
 \rho_E(\cas_\h)\epsilon&=\rho_L(\cas_\fg)\epsilon
\end{align}
where the trace in both the Casimirs is taken over $G$.

Using \eqref{caseq}  we can reformulate the infinitesimal deformation space of the canonical connection. Since the Casimir operator acts as scalar multiple of the identity on irreducible representations we can solve \eqref{caseq} for irreducible subrepresentations of $L$. From Theorem \ref{kernelelliptic} the deformations of the canonical connection are the $-2$ eigenfunctions $\epsilon\cdot\eta$ of $D^{-1,can}$. To explicitly compute the deformation space first we need to find the solutions for \eqref{caseq} which by above proposition is identical to the space of $\frac{49}{9}$ eigenfunctions $\epsilon\cdot\eta$ of $(D^{-1/3,can})^2$  . For $\alpha\in\Lambda^1\textup{Ad}_\P$ by Lemma \ref{evalues} $D^{t,A}\alpha\cdot\eta = D^{0,A}\alpha\cdot\eta+\frac{t}{2}\g2\cdot\alpha\cdot\eta=D^{0,A}\alpha\cdot\eta-\frac{t}{2}\alpha\cdot\eta $. Therefore the $\pm \frac{7}{3}$ eigenfunctions $\epsilon\cdot\eta$ of $D^{-1/3,can}$ correspond to the $-2$ and $\frac{8}{3}$ eigenfunction of $D^{-1,A}$ respectively. By Proposition \ref{kerprop} we have the following decomposition
 \begin{align}
 \begin{split}
     \ker\left((D^{-1/3,can})^2-\frac{49}{9}\id\right)\cap \Gamma(\m^*\otimes E)=&\ker(D^{-1,can}+2\id) \cap \Gamma(\m^*\otimes E) \label{kerrelation} \\ & \ker(D^{-1,can}-\frac{8}{3}\id)\cap \Gamma(\m^*\otimes E)
     \end{split}
 \end{align}
 The first summand on the right hand side is isomorphic to the space of infinitesimal deformations of $\del^{can}$ by Theorem \ref{kernelelliptic}. So in the second step we check which of the subspaces in $\ker((D^{-1/3,can})^2-\frac{49}{9}\id)\cap(\Gamma(\m^*\otimes E)\cdot\eta)$ lie in the $-2$ eigenspace of $D^{-1,can}$. 
 
The Killing spinor $\eta$ is parallel with respect to $\del^{-1}$ therefore by the definition of the Dirac operator and Proposition \ref{Dsquare} we can restrict $D^{-1,can}$ and $(D^{-1/3,can})^2$ to operators from $\Gamma(\m^*\otimes E)\to\Gamma(\m^*\otimes E)$. On a homogeneous space we can explicitly compute the canonical connection as we describe below.
\medskip 

 \noindent
\textbf{Step 1}: Calculating $\ker((D^{-1/3,can})^2-\frac{49}{9}\id)\cap\Gamma(\m^*\otimes E)$ :

Let $E_\C=\oplus_{i=1}^n V_i$ be the decomposition of $E_\C$ into complex irreducible $H$-representations. For each $V_i$ we find all the complex irreducible $G$-representations $W_{i,j}, j=1\dots n_i$, that satisfy the equation
\begin{align*}
    \rho_{V_i}(\cas_\h)&= \rho_{W_{i,j}}(\cas_\fg).
 \end{align*} 
 In order to see whether $W_{i,j}\subset \textup{Ind}^G_H(\m^*\otimes E)_\C$ we find the  multiplicity $m_{i,j}$ of $W_{i,j}$ in ${\textup{Ind}}_{H}^{G}(\m_\C^*\otimes V_i)$. Because of Schur's Lemma this multiplicity is given by  $\textup{dim}(\textup{Hom}(W_{i,j},\m_\C^*\otimes V_i)_{H})$. Repeating this process for all the $i,j$'s and summing over all irreducible $G$-representations $W_{i,j}$ along with their multiplicity we get, \begin{align}
  \ker((D^{-1/3,can})^2-\frac{49}{9}\id)\cap\Gamma(\m^*\otimes E)_\C&\cong  \bigoplus_{i=1}^n\Big( \bigoplus_{j=1}^{n_i}m_{i,j}W_{i,j}\Big).
 \end{align}
  
 \noindent
  \textbf{Step 2}: Calculating $\ker(D^{-1,can}+2\id)\cap\Gamma(\m^*\otimes E)$ :
  
  To figure out which of the $W_{i,j}$'s found in Step 1 are in the $\ker(D^{-1,can}+2\id)$ we need to calculate the covariant derivative $\del^{can}$ on $\textup{Hom}(W_{i,j},\m_\C^*\otimes V_i)_{H}\otimes W_{i,j}\subseteq \Gamma(\m^*\otimes E)_\C$. 
 
 If $(W,\tau)$ is an irreducible $G$-subrepresentation of ${\textup{Ind}}_{H}^{G}(\m^*\otimes E)$ then $\textup{Hom}(W,\m^*\otimes E)_{H}$ is non-trivial. By Schur's Lemma the dimension of $\textup{Hom}(W,\m^*\otimes E)_{H}$ is the number of common irreducible $H$-subrepresentations in $\mathrm{Res}_G^H W$ and $\m^*\otimes E$. Let $W_\alpha$ be such a common irreducible $H$-representation. We denote by $V|_{U}$ the subspace of $V$ isomoprohic to $U$ then  $\textup{Hom}(W|_{W_\alpha},(\m^*\otimes E)|_{W_\alpha}=\textup{Span}\{\phi_\alpha\}$. Let $\tau_*$ be the Lie algebra $\mathfrak{g}$ representation  associated to the $G$-representation $(W,\tau)$ then for $X\in\Gamma(TM)$ and $(\phi=\sum c_\alpha\phi_\alpha,w)\in \textup{Hom}(W,\m^*\otimes E)_{H}\otimes W$, \eqref{covariantcan}  
 \begin{align*}
     \del^{can}_X(\phi,w)(eH)&=-\phi(\tau_*(X)w)\in \m^*\otimes E.
 \end{align*}  
  Using this we can calculate the Dirac operator at $eH$ by \begin{align}\label{dcan}
     D^{-1,can}(\phi_\alpha,w)(eH)&= -\sum_{i=1}^7 e_i\cdot\del^{-1,can}_{e_i}(\phi_\alpha,w)(eH) = -\sum_{i=1}^7e_i\cdot\phi_\alpha(\tau_*(e_i)w). 
 \end{align} The above method can be extended by linearity to compute the Dirac operator on $\Gamma(\m^*\otimes{E})$. Note that we have omitted the Killing spinor $\eta$ since it is parallel with respect to $\eta$ so does not effect the eigenspace. 
 
 In the following sections we implement the above procedure on each of the four homogeneous spaces.

\begin{rem}\label{rem:difference}
In a nearly K\"ahler $6$-manifold whose structure is defined by a real Killing spinor $\eta$, the spinor $ \textup{vol}\cdot\eta$ is another independent real Killing spinor. Any Dirac operator $\mathbf{\mathcal{\slashed{D}}}$ anti-commutes with the Clifford multiplication by $\vol$ that is   $\mathbf{\mathcal{\slashed{D}}} \textup{vol} =-\mathbf{\vol\cdot\mathcal{\slashed{D}}}$, hence for all $\lambda\in\R$ we have  $\ker(\mathbf{\mathcal{\slashed{D}}}-\lambda\id)\cong \ker(\mathbf{\mathcal{\slashed{D}}}+\lambda\id)$. Therefore $\ker(\mathbf{\mathcal{\slashed{D}}}^2-\lambda^2\id)\cong 2\ker(\mathbf{\mathcal{\slashed{D}}}\pm\lambda\id)$ and one can compute the $\lambda$ eigenspace of $\slashed{D}$ by computing the $\lambda^2$ eigenspace of $\slashed{D}^2$ as done in \cite{bendef}*{Proposition 4}. In the case of nearly $\G2$ manifolds $\mathcal{\slashed{D}}$ and the $7$-dimensional $\vol$ commute and thus we do not have such an isomorphism between the $\pm\lambda$ eigenspaces of the Dirac operator. In fact there is no such automatic  relation between $\ker(\mathbf{\mathcal{\slashed{D}}}^2-\lambda^2\id)$ and $\ker(\mathbf{\mathcal{\slashed{D}}}+\lambda\id)$ as \textsection \ref{eigendirac} reveals.
\end{rem}

\begin{rem}
The Dirac operator is always self-adjoint therefore the above method of finding a particular eigenspace of a Dirac operator $D$ can be used more generally in any bundle associated to the spinor bundle over a homogeneous spin manifold. Often times it is easier to find the eigenspaces of the square of the Dirac operator $D^2$ similar to the case in hand. Once we know the $\lambda^2$-eigenspace of $D^2$ we can apply $D$ on them to see which of them lie in the $\lambda$ or $-\lambda$-eigenspace of $D$.
\end{rem}

\subsection{Eigenspaces of the square of the Dirac operator}\label{evdsquare}

In this section we follow \textit{Step 1} of the above procedure. To see which of the irreducible representations of $G$ satisfy \eqref{caseq}, we need to compute the Casimir operator on complex irreducible representations. Given any irreducible representation $\rho_\lambda$ with highest weight $\lambda$ we use the Freudenthal formula to compute $\rho_\lambda(\cas_\fg)$. We drop the constant $\frac{40}{3}$ in our definition of Casimir operator for this section as it does not play any role in comparing the Casimir operators. Let $\mu = \frac{1}{2}\text{(sum of the positive roots of} \ \fg)$ then the Freudenthal formula states that
\begin{align}\label{fruedenthal}
\rho_\lambda (\cas_\fg) = B(\lambda,\lambda)+2B(\mu,\lambda).
\end{align} 

We compute the deformation space of the canonical connection for $E=\h$ and $E=\fg_2$ as described earlier. In all the examples listed below, Case 1 is for $E=\h$ and Case 2 is for $E=\fg_2$. 

\subsubsection{$\mathrm{Spin}(7)/\G2$}

For this space, $H=\G2$ so there is only one case to consider.

The adjoint representation $\fg_2$ is the unique 14-dimensional irreducible representation of $\G2$. The complex irreducible representations of $\G2$ are identified with respect to their highest weights of the form $(p,q) \in \Z^2_{\geq 0}$ and are denoted by $V_{(p,q)}$. Here $V_{(1,0)}$ is the $7$-dimensional standard $\G2$-representation and $V_{(0,1)}$ is the $14$-dimensional adjoint representation. The reductive splitting of the Lie algebra is given by \[\mathfrak{spin}(7)  = \fg_2 \oplus \m.\]
We have the following isomorphisms of $\G2$ representations,
\begin{align*}
\h_\C &= (\fg_2)_\C \cong V_{(0,1)}\\
\m_\C &\cong V_{(1,0)}.
\end{align*}

The isomorphism $\mathfrak{spin}(7) \cong \mathfrak{so}(7)$ implies that the eigenvalues of their Casimir operators on irreducible representations are equal. 
For $\mathfrak{so}(7)$, let $E_{ij}$ be the $7\times 7$ skew-symmetric matrix with $1$ at the $(i,j)$th entry and $0$ elsewhere. We define $H_1=E_{45}-E_{23},H_2=E_{67}-E_{45}$ and $H_3=E_{45}$. A Cartan subalgebra for $\mathfrak{so}(7)$ is given by $\text{Span} \{H_i,i=1,2,3\}$. A set of simple roots $\{\alpha_i,i=1,2,3\}$ is given by 
\begin{align*}
    \alpha_1&=\begin{bmatrix}
    i\\-2i\\i
    \end{bmatrix},\ \  \alpha_2=\begin{bmatrix}
    0\\i\\-i
    \end{bmatrix}, \ \  \alpha_3=\begin{bmatrix}
    0\\i\\0
    \end{bmatrix}.
\end{align*}
 The Cartan matrix $C$ of $\mathfrak{so}(7)$ which is given by
\begin{align*}
   C&= \begin{bmatrix}
    2&-1&-1\\
    -1&2&0\\
    -2&0&2
    \end{bmatrix}.
\end{align*}
Then one can compute the simple co-roots $F_i$s by $\alpha_i(F_j)=C_{ij}$ which give $F_1=iH_2,F_2=-iH_1+2iH_3$ and $F_3=-2iH_2-2iH_3$. The set of fundamental weights is dual to the set of the simple co-roots. We denote the fundamental weights in decreasing order by $\lambda_1,\lambda_2$ and $\lambda_3$ which are dual to $F_3,F_1,F_2$ respectively. We can compute easily that
\begin{align*}
    \begin{bmatrix}
    B(H_1,H_1)& B(H_1,H_2)& B(H_1,H_2)\\
     B(H_2,H_1)& B(H_2,H_2)& B(H_2,H_3)\\
      B(H_3,H_1)& B(H_3,H_2)& B(H_3,H_3)
    \end{bmatrix}= \begin{bmatrix} -20&10&-10
\\ 10&-20&10\\ -10&10&-10
\end{bmatrix}
\end{align*}
which implies,
\begin{align*}
    \begin{bmatrix}
    B(\lambda_1,\lambda_1)& B(\lambda_1,\lambda_2)& B(\lambda_1,\lambda_2)\\
     B(\lambda_2,\lambda_1)& B(\lambda_2,\lambda_2)& B(\lambda_2,\lambda_3)\\
      B(\lambda_3,\lambda_1)& B(\lambda_3,\lambda_2)& B(\lambda_3,\lambda_3)
    \end{bmatrix}= \begin{bmatrix} 3/40&1/10&1/20
\\ 1/10&1/5&1/10\\ 1/20&1/10&
1/10\end{bmatrix}.
\end{align*}
Since half the sum of positive roots is given by $\lambda_1+\lambda_2+\lambda_3$ in \cite{Humphreys}*{Section 13.3} therefore by \eqref{fruedenthal} on an irreducible $\mathrm{SO}(7)$-representation $V_{(m_1,m_2,m_3)}$ with highest weight $m_1\lambda_1+m_2\lambda_2+m_3\lambda_3, \ m_1, m_2 , m_3\geq 0$ we have
\begin{align*}
\rho_\lambda(\cas_{\mathfrak{so}(7)})&= \frac{1}{40}(3m_1^2+8m_2^2+4m_3^2+8m_1m_2+4m_1m_3+8m_2m_3+ 18m_1+32m_2+20m_3).
\end{align*}

Now we compute the eigenvalues of the Casimir operator for the irreducible representations of $\fg_2\subset\mathfrak{so}(7)$. A Cartan subalgebra of $\fg_2$ is given by $\text{Span}\{H_1,H_2\}$. 
Here a pair of simple roots $\beta_1,\beta_2$ is given by
\begin{align*}
    \beta_1&=\begin{bmatrix}
    i\\-2i
    \end{bmatrix},\ \ \beta_2=\begin{bmatrix}
    0\\i
    \end{bmatrix}
\end{align*} 
and the Cartan matrix $\tilde{C}$ for $\lieg2$ is given by
\begin{align*}
    \tilde{C}=\begin{bmatrix}
    2&-1\\
    -3&2
    \end{bmatrix}
\end{align*}.
Let $\mu_1,\mu_2$ be the fundamental weights in decreasing order then their duals with respect to $B$ are $-iH_1-2iH_2,iH_2$ respectively and one can compute \begin{align*}
\begin{bmatrix}
B(\mu_1,\mu_1)&B(\mu_1,\mu_2)\\B(\mu_2,\mu_1)&B(\mu_2,\mu_2)
\end{bmatrix}&= \begin{bmatrix}
1/15&1/10\\ 1/10&1/5
\end{bmatrix}.
\end{align*}
 Again half the sum of the positive roots is given by $\mu_1+\mu_2$. Using these values in the Freudenthal formula for an irreducible $\G2$-representation $V_{(p,q)}$ with highest weight $p\mu_1+q\mu_2$ we have\begin{align*}
\rho_{(p,q)}(\cas_{\fg_2})&=\frac{1}{15}(p^2+3q^2+3pq+5p+9q).
\end{align*}

\noindent
\textit{Case 1: $E= \fg_2$}

\noindent
The adjoint representation $(\lieg2)_\C\cong V_{(0,1)}$. From above
\begin{align*}
 \rho_{(0,1)}(\cas_{\fg_2})& =\frac{4}{5}.
\end{align*}  
Substituting the above found values into \eqref{caseq} we get that $V_{(m_1,m_2,m_3)}$ can be an infinitesimal deformation space for the canonical connection if \begin{align*}
\frac{1}{40}(3m_1^2+8m_2^2+4m_3^2+8m_1m_2+4m_1m_3+8m_2m_3+ 18m_1+32m_2+20m_3)&=\frac{4}{5}.
\end{align*}But since there are no positive integral solutions of this equation there are no deformations of the canonical connection on $\mathrm{Spin}(7)/\G2$.

\subsubsection{$\mathrm{SO(5)/SO(3)}$}
 The complex irreducible $\mathrm{SO}(5)$-representations are characterized by highest weights $(m_1,m_2)\in\mathbb{Z}_{\geq 0}$. The complex irreducible representations of $\mathrm{SO}(3)$ are given by $S^k\C^2$ which is a $\binom{2+k-1}{k} =k+1$ dimensional space. The $3$-dimensional adjoint representation $\mathfrak{so}(3)_\C$ and the $7$-dimensional representation $\m_\C$ are irreducible $\mathrm{SO}(3)$-representations therefore \begin{align*}
\m_\C &\cong S^6\C^2,\\
\mathfrak{so}(3)_\C &\cong S^2\C^2.
\end{align*}

A Cartan subalgebra of $\mathfrak{so}(5)$ is given by $\text{Span}\{H_1,H_2\}$ where $H_1=E_{12},H_2=E_{34}$ where $E_{ij}$ is the $5\times 5$ skew-symmetric matrix with $1$ at the $(i,j)$th position and $0$ elsewhere. With respect to the Killing form $B$ on $\mathfrak{so}(5)$, $H_1$ is orthogonal to $H_2$  with $B(H_i,H_i)=-6$ for $i=1,2$. Let $\lambda_1,\lambda_2$ be the fundamental weights whose duals are $i(H_1-H_2),2iH_2$ respectively then half the sum of positive roots is given by $\lambda_1+\lambda_2$.  Doing similar computations as above we get \begin{align*}
    \begin{bmatrix}
    B(\lambda_1,\lambda_1)&B(\lambda_1,\lambda_2)\\B(\lambda_2,\lambda_1)&B(\lambda_2,\lambda_2)
    \end{bmatrix}= \begin{bmatrix}
    1/6&1/12\\1/12&1/12
    \end{bmatrix}.
\end{align*} Using \eqref{fruedenthal} for the eigenvalues of the Casimir operator for irreducible representation $V_{(m_1,m_2)}$ of $\mathrm{SO}(5)$ with highest weight $m_1\lambda_1+m_2\lambda_2$ for $m_1,m_2\geq 0$   we get,\begin{align*}
\rho_{(m_1,m_2)}(\cas_{\mathfrak{so}(5)})&= \frac{1}{12}(2m_1^2+m_2^2+2m_1m_2+6m_1+4m_2).
\end{align*}

Under the embedding of $\mathfrak{so}(3)$ in $\mathfrak{so}(5)$ the  Cartan subalgebra of $\mathfrak{so}(3)$ is given by $\text{Span}\{2H_1+H_2\}$. Here the Cartan subalgebra is $1$-dimensional and the fundamental weight $\mu_1$ is dual to $4iH_1+2iH_2$. Using $B(H_i,H_i)=-6$ one can compute that $B(4H_1+2H_2,4H_1+2H_2)=-120$ the eigenvalue of the Casimir operator on the irreducible representation $S^q\C^2$ of $\mathfrak{so}(3)$ is given by \begin{align*}
\rho_q(\cas_{\mathfrak{so}(3)}) =\frac{1}{120}(q^2+2q).
\end{align*}

\noindent
\textit{Case 1: $E=\mathfrak{so}(3)$}

\noindent
The adjoint representation of $\mathfrak{so}(3)_\C$ is an irreducible $\mathfrak{so}(3)$ representation with highest weight $2$. Thus \begin{align*}
\rho_{E}(\cas_{\mathfrak{so}(3)})&=\rho_{2}(\cas_{\mathfrak{so}
(3)})=\frac{1}{15}.
\end{align*}
We need to find irreducible representations $V_{(m_1,m_2)}$ of $\mathfrak{so}(5)$ that satisfy \eqref{caseq} which requires \begin{align*}
\frac{1}{12}(2m_1^2+m_2^2+2m_1m_2+6m_1+4m_2)&=\frac{1}{15}.
\end{align*}  But since there are no integral solutions for the equation the deformation space is trivial in this case.

\medskip

\noindent
\textit{Case 2: $E=\mathfrak{\fg_2}$}

\noindent
The adjoint representation of $(\fg_2)_\C$ splits as an $\mathfrak{so}(3)$ representation  into $S^2\C^2\oplus S^{10}\C^2$. The first component in the splitting has already been studied in case 1 and hence has no contribution to the deformation space.
For the second component \begin{align*}
\rho_{10}(\cas_{\mathfrak{so}(3)})=1.
\end{align*} Thus we need to find $\mathfrak{so}(5)$ representations $V_{(m_1,m_2)}$ such that \begin{align*}
\frac{1}{12}(2m_1^2+m_2^2+2m_1m_2+6m_1+4m_2)&=1,
\end{align*} which has one integral solution namely $m_1=0,m_2=2$. Thus $V_{(0,2)}\cong \mathfrak{so}(5)_\C$ is the only SO(5)-representation for which $\cas_\fg$ has eigenvalue $1$. As $\mathfrak{so}(3)$ representations \begin{align*}
    V_{(0,2)}\cong S^2\C^2\oplus S^6\C^2,\\
    \m_\C^*\otimes S^{10}\C^2\cong \bigoplus_{k=2}^8 S^{2k}\C^2.
\end{align*}Thus $V_{(0,2)}$ and $\m_\C^*\otimes S^{10}\C^2$ have $1$ common irreducible $\mathfrak{so}(3)$ representation namely $S^6\C^2$. Thus $V_{(0,2)}$ occurs in $\textup{Ind}_{H}^G(\m_\C^*\otimes S^{10}\C^2)$ with multiplicity $1$. Therefore in this case $(\ker((D^{-1/3,can})^2-\frac{49}{9}\id)\cap\Gamma(\m^*\otimes\lieg2))_\C\cong V_{(0,2)}$. 

\subsubsection{ $\frac{\mathrm{Sp}(2)\times\mathrm{Sp}(1)}{\mathrm{Sp}(1)\times \mathrm{Sp}(1)}$}\label{Sp2Sp1}

The Lie algebra $\mathfrak{sp}(2)\oplus\mathfrak{sp}(1)$ decomposes as \begin{align*}
\mathfrak{sp}(2)\oplus\mathfrak{sp}(1)&= \mathfrak{sp}(1)_u\oplus \mathfrak{sp}(1)_d \oplus \m
\end{align*} and the embeddings $\mathfrak{sp}(1)_u,\mathfrak{sp}(1)_d$ are given by \begin{align*}
\mathfrak{sp}(1)_u&=\Big\{ \Big(\begin{pmatrix}
a&0\\0&0
\end{pmatrix},0\Big) : a\in \mathfrak{sp}(1)\Big\}, \ \ \ \mathfrak{sp}(1)_d =\Big\{ \Big(\begin{pmatrix}
0&0\\0&a
\end{pmatrix},a\Big) : a\in \mathfrak{sp}(1)\Big\}
\end{align*} where we follow the notations used in \cite{deformg2}. Let $H_1=(E_1,0),H_2=(E_2,0)$ and $H_3=(0,E_3)$ then a Cartan subalgebra of $\mathfrak{sp}(2)\oplus\mathfrak{sp}(1)$ is given by $\text{Span}\{H_1,H_2,H_3\}$ where \begin{align*}
E_1&= \begin{pmatrix}
i&0&0&0\\0&0&0&0\\0&0&-i&0\\0&0&0&0
\end{pmatrix}, E_2=\begin{pmatrix}
0&0&0&0\\0&i&0&0\\0&0&0&0\\0&0&0&-i
\end{pmatrix}, E_3=\begin{pmatrix}
i&0\\0&-i
\end{pmatrix}.
\end{align*} If $B$ denote the Killing form of $\mathrm{Sp}(2)\times\mathrm{Sp}(1)$ we can compute that $H_i$s are orthogonal with respect to $B$ and $B(H_i,H_i)=-12$ for $i=1,2$ and $B(H_3,H_3)=-8$. The  fundamental weights $\lambda_1,\lambda_2,\lambda_3$ are dual to $i(H_1-H_2),iH_1,iH_3$ respectively and half the sum of positive roots is given by $\lambda_1+\lambda_2+\lambda_3$. By identical calculations as in other cases we get 
\begin{align*}
\begin{bmatrix}
B(\lambda_1,\lambda_1)&B(\lambda_1,\lambda_2)&B(\lambda_1,\lambda_3)\\
B(\lambda_2,\lambda_1)&B(\lambda_2,\lambda_2)&B(\lambda_2,\lambda_3)\\
B(\lambda_3,\lambda_1)&B(\lambda_3,\lambda_2)&B(\lambda_3,\lambda_3)
\end{bmatrix} &= \begin{bmatrix}
1/12&1/12&0\\1/12&1/6&0\\0&0&1/8
\end{bmatrix}.
\end{align*} Applying the Freudenthal formula \eqref{fruedenthal} we get that the Casimir operator of $\mathfrak{sp}(2)\oplus\mathfrak{sp}(1)$ acts on the irreducible representations $V_{(m_1,m_2,l)}$ with highest weight $m_1\lambda_1+m_2\lambda_2+l\lambda_3, m_1,m_2,l\geq 0$ with the eigenvalue \begin{align*}
\rho_{(m_1,m_2,l)}(\cas_{\mathfrak{sp}(2)\oplus\mathfrak{sp}(1)})&= \frac{1}{12}(m_1^2+2m_2^2+2m_1m_2+4m_1+6m_2)+\frac{1}{8}(l^2+2l).
\end{align*}

Under the embedding given above a Cartan subalgebra of $\mathfrak{sp}(1)_u,\mathfrak{sp}(1)_d$ is given by $\text{Span}\{H_1\}$ and $\text{Span}\{(E_2,E_3)\}$ respectively. Let $P,Q$ be the standard $2$-dimensional representation of $\mathfrak{sp}(1)_u,\mathfrak{sp}(1)_d$ respectively. Then the unique $(n+1)-$dimensional irreducible $\mathfrak{sp}(1)_u$ (respectively $\mathfrak{sp}(1)_d$) representation is given by $S^nP$ (respectively $S^nQ$). 
From previous calculations we have $B(H_1,H_1)=-12$ thus the eigenvalue of $\cas_{\mathfrak{sp}(1)_u}$ on  $S^nP$ is given by \begin{align*}
\rho_{n}(\cas_{\mathfrak{sp}(1)_u})&=\frac{1}{12}(n^2+2n).
\end{align*}
Similarly with the help of previous work one can calculate $B((E_2,E_3),(E_2,E_3))=-20$. Thus $\cas_{\mathfrak{sp}(1)_d}$ acts on $S^pQ$ as the scalar multiple of \begin{align*}
\rho_{n}(\cas_{\mathfrak{sp}(1)_d})&=\frac{1}{20}(n^2+2n).
\end{align*} 

The adjoint representation $\mathfrak{sp}(1)$ is an irreducible $3$-dimensional $\mathfrak{sp}(1)$ representation and hence we have the following decompositions into $\mathrm{Sp}(1)_u\times \mathrm{Sp}(1)_d$ representations \begin{align*}
(\mathfrak{sp}(1)_u)_\C &\cong S^2P, \ \ \ (\mathfrak{sp}(1)_d)_\C \cong S^2Q, \ \ \ \m_\C\cong S^2Q\oplus PQ
\end{align*} where $PQ$ denotes the tensor product of $P$ and $Q$ and we omitted the tensor product sign for clarity and will continue to do so. 

\medskip
\noindent
\textit{Case 1: $E=\mathfrak{sp}(1)_u\oplus \mathfrak{sp}(1)_d$}

\noindent
We need to find the irreducible $\mathfrak{sp}(2)\oplus\mathfrak{sp}(1)$ representations $V_{(m_1,m_2,l)}$ that satisfy \eqref{caseq} for each irreducible component of $\h_\C$ that is $(\mathfrak{sp}(1)_u)_\C$ and $(\mathfrak{sp}(1)_d)_\C$ . For $\mathfrak{sp}(1)_u$ this equation takes the form \begin{align*}
\frac{1}{12}(m_1^2+2m_2^2+2m_1m_2+4m_1+6m_2)+\frac{1}{8}(l^2+2l)&= \frac{8}{12}.
\end{align*} The integral solution $(m_1,m_2,l)$ for this equation is $(0,1,0)$. Thus the only irreducible  $\mathfrak{sp}(2)\oplus\mathfrak{sp}(1)$ representations for which $\cas_\fg$ has eigenvalue $\frac{2}{3}$ is $ V_{(0,1,0)}$. As $\mathfrak{sp}(1)_u\oplus\mathfrak{sp}(1)_d$-representations we have the following decomposition \begin{align*}
V_{(0,1,0)} &\cong  PQ\oplus \C,\\
(\mathfrak{sp}(1)_u\otimes \m)_\C &\cong  S^2P S^2Q\oplus S^3PQ\oplus PQ.
\end{align*} The irreducible $\mathrm{Sp}(1)\times\mathrm{Sp}(1)$ representation in $(\mathfrak{sp}(1)_u\otimes \m)_\C$ common with $V_{(0,1,0)}$  is $PQ$ with multiplicity $1$. Thus $ V_{(0,1,0)}$ occurs in $\textup{Ind}_{H}^G(\m^*\otimes{\mathfrak{sp}(1)_u})_\C$ with multiplicity $1$. Therefore the solutions to \eqref{caseq} in $\Gamma(\mathfrak{{\m^*}\otimes sp(1)}_u )_\C$ is the $5$-dimensional complex $\mathrm{Sp}(2)\times\mathrm{Sp}(1)$ representation $V_{(0,1,0)}$.

For the next irreducible $\h_\C$ component $(\mathfrak{sp}(1)_d)_\C$ \eqref{caseq} for $V_{(m_1,m_2,l)}$ becomes \begin{align*}
\frac{1}{12}(m_1^2+2m_2^2+2m_1m_2+4m_1+6m_2)+\frac{1}{8}(l^2+2l)&= \frac{8}{20},
\end{align*} which has no integral solutions and thus it has no contribution to the deformation space. 

Thus from Proposition \ref{casprop} we conclude that $(\ker((D^{-1/3.can})^2-\frac{49}{9}\id)\cap\Gamma(\m^*\otimes\mathfrak{sp}(1)_u\oplus\mathfrak{sp}(1)_d)_\C \cong  (V_{(0,1,0)})$ when the structure group is $\mathrm{Sp}(1)_{u}\times\mathrm{Sp}(1)_{d}$.

\medskip
\noindent
\textit{Case 2: $E=(\fg_2)_\C$}

\noindent
The adjoint representation of $\fg_2$ decomposes into irreducible $\mathfrak{sp}(1)_u\oplus \mathfrak{sp}(1)_d$ as follows: \begin{align*}
(\fg_2)_\C &= S^2P\oplus S^2Q \oplus PS^3Q.
\end{align*} We have already seen the contribution of the first two irreducible components in the summation. For the third component \begin{align*}
\rho_{1,3}(\cas_{\mathfrak{sp}(1)_u\oplus \mathfrak{sp}(1)_d})&= 1,
\end{align*} so here we need to find the $\mathfrak{sp}(2)\oplus\mathfrak{sp}(1)$ representations $V_{(m_1,m_2,l)}$ such that \begin{align*}
\frac{1}{12}(m_1^2+2m_2^2+2m_1m_2+4m_1+6m_2)+\frac{1}{8}(l^2+2l)&= 1. 
\end{align*} The $\mathfrak{sp}(2)\oplus\mathfrak{sp}(1)$-representations that satisfy \eqref{caseq} are $V_{(2,0,0)}$ and  $V_{(0,0,2)}$, which decompose into $\mathfrak{sp}(1)_u\oplus \mathfrak{sp}(1)_d$ representations as \begin{align*}
V_{(2,0,0)} &\cong \mathfrak{sp}(2)_\C\cong S^2P\oplus S^2Q\oplus PQ, \ \ \ 
V_{(0,0,2)}\cong (\mathfrak{sp}(1)_d)_\C\cong S^2Q.
\end{align*} Moreover \begin{align*}
PS^3Q \otimes \m_\C^* &\cong S^2PS^4Q\oplus S^2PS^2Q\oplus P(S^5Q\oplus S^3Q\oplus Q)\oplus S^4Q\oplus S^2Q. 
\end{align*} Thus $V_{(2,0,0)}$ and $PS^3Q\otimes\m_\C^*$ have two common irreducible representations $PQ,S^2Q$ and $V_{(0,0,2)}$ and $PS^3Q\otimes\m_\C^*$ have one common irreducible representation $S^2Q$. So by Frobenius reciprocity $V_{(2,0,0)}$ and $V_{(0,0,2)}$ lie in $\textup{Ind}_{H}^G(\m_\C^*\otimes{PS^3Q})$ with multiplicity $2,1$ respectively. Thus the solution of \eqref{caseq} in $\Gamma( \m^*\otimes \fg_2)_\C$ is  the $28$ dimensional  $\mathrm{Sp}(2)\times\mathrm{Sp}(1)$ complex representation $2V_{(2,0,0)}\oplus V_{(0,1,0)}\oplus V_{(0,0,2)}$. So again by Proposition \ref{casprop} we conclude that $\ker((D^{-1/3.can})^2-\frac{49}{9}\id)\cap\Gamma(\m^*\otimes \lieg2)_\C\cong 2V_{(2,0,0)}\oplus V_{(0,1,0)}\oplus V_{(0,0,2)}$ when the structure group is $\G2$.

\subsubsection{ $\frac{\mathrm{SU}(3)\times\mathrm{SU}(2)}{\mathrm{SU}(2)\times \mathrm{U}(1)}$}

The embeddings of $\mathfrak{su}(2)$ and $\mathfrak{u}(1)$ in $\mathfrak{su}(3)\times\mathfrak{su}(2)$ which we denote by $\mathfrak{su}(2)_d$ and $\mathfrak{u}(1)$  following \cite{deformg2} in $\mathfrak{su}(3)\oplus\mathfrak{su}(2)$ are given by  \begin{align*}
\mathfrak{su}(2)_d&= \Big\{ \Big(\begin{pmatrix}
a&0\\0&0
\end{pmatrix},a\Big): a\in \mathfrak{su}(2)\Big\}, \ \ \ \mathfrak{u}(1)=\text{span}\Big\{\Big( \begin{pmatrix}
i&0&0\\0&i&0\\0&0&-2i
\end{pmatrix},0\Big)\Big\}.
\end{align*}A Cartan subalgebra of $\mathfrak{su}(3)\oplus\mathfrak{su}(2)$ is given by $\text{span}\{H_1=(E_1,0),H_2=(E_2,0),H_3=(0,E_3)\}$ where \begin{align*}
E_1&= \begin{pmatrix}
0&1&0\\-1&0&0\\0&0&0
\end{pmatrix}, \ \ E_2 =\begin{pmatrix}
i&0&0\\0&i&0\\0&0&-2i
\end{pmatrix}, \ \ E_3= \begin{pmatrix}
0&1\\-1&0
\end{pmatrix}.
\end{align*} We can check that the $H_i$s are orthogonal with respect to the Killing form $B$ on $\mathrm{SU}(3)\times\mathrm{SU}(2)$. As earlier we denote by $\lambda_1,\lambda_2,\lambda_3$ the fundamental weights which are dual to $\frac{i}{2}(H_1-H_2),\frac{i}{2}(H_1+H_2),iH_3$ respectively. By direct computations we get \begin{align*}
\begin{bmatrix}
B(H_1,H_1)&B(H_1,H_2)&B(H_1,H_3)\\
B(H_2,H_1)&B(H_2,H_2)&B(H_2,H_3)\\
B(H_3,H_1)&B(H_3,H_2)&B(H_3,H_3)
\end{bmatrix} &= \begin{bmatrix}
-12&0&0\\0&-36&0\\0&0&-8
\end{bmatrix},
\end{align*} therefore \begin{align*}
\begin{bmatrix}
B(\lambda_1,\lambda_1)&B(\lambda_1,\lambda_2)&B(\lambda_1,\lambda_3)\\
B(\lambda_2,\lambda_1)&B(\lambda_2,\lambda_2)&B(\lambda_2,\lambda_3)\\
B(\lambda_3,\lambda_1)&B(\lambda_3,\lambda_2)&B(\lambda_3,\lambda_3)
\end{bmatrix} &= \begin{bmatrix}
1/9&1/18&0\\1/18&1/9&0\\0&0&1/8
\end{bmatrix}.
\end{align*} Half the sum of the positive roots is $\lambda_1+\lambda_2+\lambda_3$ and thus by Freudenthal formula \eqref{fruedenthal} for a $\mathfrak{su}(3)\oplus\mathfrak{su}(2)$ representation $V_{(m_1,m_2,l)}$ with highest weight $m_1\lambda_1+m_2\lambda_2+l\lambda_3$ where $m_1,m_2,l\geq 0$ \begin{align*}
\rho_{m_1,m_2,l}(\cas_{\mathfrak{su}(3)\oplus\mathfrak{su}(2)})&= \frac{1}{9}(m_1^2+m_2^2+m_1m_2+3m_1+3m_2)+\frac{1}{8}(l^2+2l).
\end{align*}

Using the embeddings of $\mathfrak{su}(2)$ and  $\mathfrak{u}(1)$ given above we see that Cartan subalgebras of  $\mathfrak{su}(2)$ and $\mathfrak{u}(1)$ in $\mathfrak{su}(3)\oplus\mathfrak{su}(2)$ are given by $\text{span}\{(E_1,E_3)\}$ and $\text{span}\{H_2\}$ respectively. By calculations completely analogous to the previous case we then get that if we represent the irreducible $(n+1)$-dimensional $\mathfrak{su}(2)_d$ representations by $S^nW$ where $W$ is the standard $\mathfrak{su}(2)_d$ representation and the $1$-dimensional $\mathfrak{u}(1)$ representation with highest weight $k$ by $F(k)$ we get by the Freudenthal formula \eqref{fruedenthal} \begin{align*}
\rho_n(\cas_{\mathfrak{su}(2)_d})&=\frac{1}{20}(n^2+2n),\\
\rho_k(\cas_{\mathfrak{u}(1)})&=\frac{1}{36}k^2. 
\end{align*} As $\mathfrak{su}(2)_d\oplus \mathfrak{u}(1)$ representations the $7$-dimensional space $\m_\C$ decomposes as \begin{align*}
\m_\C&\cong S^2W\oplus WF(3)\oplus W F(-3),
\end{align*}whereas the  $3$-dimensional adjoint representation of $(\mathfrak{su}(2)_d)_\C$ is irreducible and hence is isomorphic to $S^2W$. 

\medskip
\noindent
\textit{Case 1: $E=\mathfrak{su}(2)_d\oplus \mathfrak{u}(1)$}

\noindent 
The adjoint representation $\mathfrak{su}(2)_d\oplus\mathfrak{u}(1)$ splits as irreducible $\mathfrak{su}(2)_d\oplus \mathfrak{u}(1)$ representations as follows:
\begin{align*}
(\mathfrak{su}(2)_d\oplus\mathfrak{u}(1))_\C&\cong S^2W\oplus \C.
\end{align*} 
Since $U(1)$ is abelian we know by Theorem \ref{abelian} that the component $\mathfrak{u}(1)$ is abelian and thus gives rise to no deformations of the canonical connection. Therefore we only need to check for deformations corresponding to $S^2W$. For that we need to look for representations $V_{(m_1,m_2,l)}$ such that \begin{align*}
\frac{1}{9}(m_1^2+m_2^2+m_1m_2+3m_1+3m_2)+\frac{1}{8}(l^2+2l)&= \frac{8}{20},
\end{align*} which as seen before has no integral solutions.

Hence the canonical connection admits no deformations in this case.

\medskip
\noindent
\textit{Case 2: $E=\fg_2$}

\noindent
The adjoint representation $(\fg_2)_\C$ splits as $\mathfrak{su}(2)_d\oplus \mathfrak{u}(1)$ representation as follows:
\begin{align*}
(\fg_2)_\C&= S^3WF(3)\oplus S^3WF(-3)\oplus S^2W\oplus F(6)\oplus F(-6)\oplus \C.
\end{align*}
We need to follow the same procedure as above for each of the components. For each component we need to find the $\mathfrak{su}(3)\oplus\mathfrak{su}(2)$ representation $V_{(m_1,m_2,l)}$ that satisfies \eqref{caseq}. We have already solved this for $S^2W\oplus\C$ so we just need to compute it for the rest.

\noindent
From above calculations $\rho_{S^3WF(3)}(\cas_\h)=1$ therefore $V_{(m_1,m_2,l)}$ should satisfy\begin{align*}
\frac{1}{9}(m_1^2+m_2^2+m_1m_2+3m_1+3m_2)+\frac{1}{8}(l^2+2l)&= 1.
\end{align*}The only possible solutions are $V_{(0,0,2)},V_{(1,1,0)}$. As $\mathfrak{su}(2)\otimes \mathfrak{u}(1)$ representations $V_{(0,0,2)}\cong S^2W$ and $V_{(1,1,0)}\cong \mathfrak{su}(3)_\C$.
Further one can compute\begin{align*}
V_{(0,0,2)}&\cong \mathfrak{su}(2)_\C\cong S^2W, \\
V_{(1,1,0)}&\cong \mathfrak{su}(3)_\C\cong S^2W\oplus  WF(3)\oplus W F(-3)\oplus\C,\\
S^3WF(3)\otimes \m_\C^* &\cong (S^5W\oplus S^3W\oplus W)F(3)\oplus (S^4W\oplus S^2W)F(6)\oplus S^4W\oplus S^2W.
\end{align*}
Thus $V_{(0,0,2)}$ and $S^3WF(3)\otimes \m_\C^*$ has one common component $S^2W$ with multiplicity $1$  and $V_{(1,1,0)}$ and $S^3WF(3)\otimes \m_\C^*$ has two common components $S^2W,WF(3)$ both with multiplicity $1$ each.
So by Frobenius reciprocity $\textup{Ind}_{H}^G(\m_\C^*\otimes S^3WF(3))$ contains a copy of $V_{(0,0,2)}\oplus 2V_{(1,1,0)}$.

 The representation $S^3WF(-3)$ is the dual of the representation $S^3WF(3)$ and since $\mathrm{SU}(2)\otimes \mathrm{U}(1)$ representations are isomorphic to their duals the result for this case is same as the above and $\textup{Ind}_{H}^G(\m_\C^*\otimes S^3WF(-3))$ also contains a copy of $V_{(0,0,2)}\oplus 2V_{(1,1,0)}$.

For the $\mathfrak{u}(1)$ representation $F(6)$, $\rho_{6}(\cas_{\mathfrak{u}(1)})=1$. Thus again the only solutions are $V_{(0,0,2)}, V_{(1,1,0)}$ by the previous case. The $\mathfrak{su}(2)\oplus\mathfrak{u}(1)$ representation $F(6)\otimes\m_\C^*$ has the following decompostion
\begin{align*}
     F(6)\otimes \m^*_\C \cong S^2WF(6)\oplus WF(9)\oplus WF(3),
 \end{align*} thus $V_{(0,0,2)}$ is not contained in $\textup{Ind}_{H}^G(\m_\C^*\otimes F(6))$ but $V_{(1,1,0)}$ is with multiplicity 1. Since $F(-6)\cong F(6)^*$ this case is similar to the above case.
 
 \noindent
 Summing up all the parts together we get that  $\ker((D^{-1/3,can})^2- \frac{49}{9}\}\cap\Gamma(\m^*\otimes\lieg2)_\C\cong 2(V_{(0,0,2)}\oplus 3V_{(1,1,0)})$ when the structure group is $\G2$.

\medskip

 Table \ref{tablekerDsquare} lists the $\ker((D^{-1/3,can})^2- \frac{49}{9}\id)\cap\Gamma(\m^*\otimes E)$ when $E=\h$ and $E=\lieg2$ for all the homogeneous spaces listed in Table \ref{tablenormal}. Note that for the remaining two homogeneous spaces $N_{k,l},k\neq l$ and $\mathrm{SU(2)}^3/\mathrm{U(1)}^2$ our methods does not apply when $E=\mathfrak{g_2}$ although since $H$ is abelian for both of them there are no deformations for the $E=\h$ case. The space $V^{(0,1)}$ listed in Table \ref{tablekerDsquare} denotes the unique irreducible $5$-dimensional complex representation of $\mathfrak{sp}(2)$.
 
\begin{table}[ht]
\centering
\begin{tabular}{c|c|c}
Homogeneous space & $\h$& $\fg_2$\\
\hline
Spin(7)/$\G2$& $0$ & $0$\\ 
& &\\
SO(5)/SO(3) &$0$ & $\mathfrak{so}(5)$\\ & &\\
$\frac{\mathrm{Sp}(2)\times \mathrm{Sp}(1)}{\mathrm{Sp}(1)\times \mathrm{Sp}(1)}$ & $ V^{(0,1)}_\R$ & $2\mathfrak{sp}(2)\oplus \mathfrak{sp}(1)\oplus V^{(0,1)}_\R$\\ & &\\
$\frac{\text{SU(3)}\times \text{SU(2)}}{\text{SU(2)}\times \text{ U(1)}}$&$0$& $2\mathfrak{su}(2)\oplus 6\mathfrak{su}(3)$\\ & &\\
N$_{k,l}$ & $0$ & unknown \\& & \\
SU(2)$^3$/U(1)$^2$& $0$ & unknown 
\end{tabular}
\caption{$\ker((D^{-1/3,can})^2-\frac{49}{9}\id)\cap\Gamma(\m^*\otimes E)$}
\label{tablekerDsquare}
\end{table}

\subsection{Eigenspaces of the Dirac operator}\label{eigendirac}                                                       

All the $G$-representations listed in Table \ref{tablekerDsquare} lie in $\ker((D^{-1/3,can})^2-\frac{49}{9}\id)\cap\Gamma(\m^*\otimes E)$ which by \eqref{kerrelation} is equal to $(\ker(D^{-1,can} +2\id)\oplus\ker(D^{-1,can} -\frac 83\id))\cap\cap\Gamma(\m^*\otimes E)$. Since the canonical connection is translation invariant it takes an irreducible $G$-representation to itself. 
Hence the irreducible subspaces found in Table~\ref{tablekerDsquare} lie in either $\ker(D^{-1,can}-\frac{8}{3}\id)$ or $ \ker(D^{-1,can}+2\id)$ where the subspaces in the latter space constitute the infinitesimal deformations of the canonical connection by Theorem \ref{kernelelliptic}. Thus now it remains to identify which of the subspaces in Table~\ref{tablekerDsquare} lies in $\ker(D^{-1,can}+2\id)$ for each of the homogeneous spaces.
for all the homogeneous spaces $G/H$ in Table~\ref{tablenormal} the metric corresponding to the nearly $\G2$ structure $\g2$ is given by $-\frac{3}{40} B$ where $B$ is the Killing form of $G$.  For $1$-forms $X,Y$ the Clifford product between $X$ and $Y\cdot\eta$  is given by\begin{align}\label{cliffprodforms}
X\cdot Y\cdot\eta&= \langle X,Y\rangle\eta - \g2(X,Y,.)\cdot\eta.
\end{align} Thus we have all the ingredients in \eqref{dcan} to calculate the action of the Dirac operator $D^{-1,can}$ on each irreducible subspace in Table~\ref{tablekerDsquare}.

\subsubsection{$\mathrm{SO}(5)/\mathrm{SO}(3)$}

From the previous section we know that there are no deformation of the canonical connection when the structure group is SO(3). For the structure group $\G2$ we calculated that the smooth sections of $G\times_{\rho_{\m^*\otimes\lieg2}}(\m^*\otimes\lieg2)$ in $\ker((D^{-1/3.can})^2-\frac{49}{9}\id)\cong V_{(0,2)}\cong \mathfrak{so}(5)_\C$. If we denote by $E_{ij}$ the skew-symmetric matrix with $1$ at $(i,j)$, $-1$ at $(j,i)$ and $0$ elsewhere and define \begin{align*}
  e_1&:=  \frac{2}{3}(E_{12}-2E_{34}),\ \ \ \ \ 
  e_2:=\frac{2}{3}(\sqrt{2}E_{45}-\frac{\sqrt 3}{\sqrt 2}(E_{23}-E_{14})), \\ e_3&:=\frac{2\sqrt{5}}{3}E_{25},\hspace{1.8cm}
  e_4:=\frac{2}{3}(\sqrt 2 E_{35}-\frac{\sqrt 3}{\sqrt 2}(E_{13}+E_{24})), \\ e_5&:=\frac{\sqrt{10}}{3}(E_{24}-E_{13}),\ \ \ e_6:=-\frac{\sqrt{10}}{3}(E_{23}+E_{14}), \ \ \ e_7:=\frac{2\sqrt{5}}{3}E_{15},
\end{align*} then $\{e_i,i=1\dots 7\}$ defines a basis of $\m^*$ which is orthonormal with respect to the metric $-\frac{3}{40}B$. With respect to this basis the nearly $G_2$ structure $\g2$ is given by \begin{align*}
    \g2&=e_{124}+e_{137}+e_{156}+e_{235}+e_{267}+e_{346}+e_{457}.
\end{align*}
We have seen that for $\mathrm{SO}(5)/\mathrm{SO}(3)$ the canonical connection has no deformation as an $\mathrm{SO}(3)$ connection.  Now we need to check whether the SO(5)-representation $V_{(0,2)}$ lies in the $\ker(D^{-1,can}-\frac{8}{3}\id)\cap\Gamma(\m^*\otimes\lieg2)_\C$ or $\ker(D^{-1,can}+2\id)\cap\Gamma(\m^*\otimes\lieg2)_\C$. 
 As seen before the common irreducible $\mathfrak{so}(3)$ representation in $V_{(0,2)}|_{\mathfrak{so}(3)}$ and $(\m^*\otimes \lieg2)_\C$ is $S^6\C^2 \cong \m^*_\C$. We denote the $1$-dimensional space $\textup{Hom}(V_{(0,2)},(\m^*\otimes \lieg2)_\C)=\textup{Span}(\alpha).$ Let $\mu_i,i=1\ldots11$ be a basis of the $11$-dimensional subspace of $(\lieg2)_\C$ isomorphic to the $\mathfrak{so}(3)$ representation $S^{10}\C^2$. Then the subspace of $\m_\C^*\otimes S^{10}\C^2\subset (\m^*\otimes\lieg2)_\C$ isomorphic to $S^6\C^2$ is given by $\textup{Span}\{v_i,i=1\dots 7\}$ where \begin{align*}
     v_1=&-\frac{e_2}{14}\otimes(5(\mu_1-\mu_7)+3\sqrt{15}\mu_9)+e_3\otimes(\mu_5+\mu_{11}) -\frac{e_4}{14}\otimes(5\mu_2+3\sqrt{15}(\mu_3+\mu_4))\\&+e_5\otimes(\mu_3-\mu_4)+e_6\otimes\mu_9+e_7\otimes(\mu_6-\mu_{10}),\\
     v_2=&e_1\otimes\mu_9+e_2\otimes(-2\mu_5+\mu_4)-\frac{e_3}{28}\otimes(47\mu_1+37\mu_7+3\sqrt{5}\mu_9)-e_4\otimes(\mu_6+2\mu_{10})\\&-\frac{e_5}{14}\otimes\mu_8+\frac{e_7}{28}\otimes(-37\mu_2+3\sqrt{15}(\mu_3+\mu_4)),\\
     v_3=&-\frac{e_1}{2}\otimes(\mu_3-\mu_4)+\frac{e_2}{2}\otimes(2\mu_6+\mu_{10})+\frac{e_3}{56}\otimes(47\mu_2+3\sqrt{5}(\mu_3+\mu_4))-\frac{e_4}{2}\otimes(\mu_5-2\mu_{11})\\&-\frac{e_6}{28}
\otimes\mu_8 +\frac{e_7}{56}(-37\mu_1+6\sqrt{15}\mu_9),\\
v_4=&-\frac{e_1}{28}\otimes(5\mu_2+3\sqrt{15}(\mu_3+\mu_4))+\frac{5e_2}{28}\otimes\mu_8-\frac{e_3}{56}\otimes(3\sqrt{15}\mu_2+41\mu_3+13\mu_4)\\&-\frac{e_5}{2}\otimes(\mu_5-2\mu_{11})+\frac{e_6}{2}\otimes(\mu_6+2\mu_{10})+\frac{e_7}{56}\otimes(3\sqrt{15}(\mu_1-\mu_7)+41\mu_9),\\
v_5=&e_1\otimes(\mu_5+\mu_{11})-\frac{e_2}{28}\otimes(3\sqrt{15}(\mu_1-\mu_7)+13\mu_9)+\frac{e_4}{28}\otimes(3\sqrt{15}\mu_2+41\mu_3+13\mu_4)\\&+\frac{e_5}{28}\otimes(47\mu_2+3\sqrt{15}(\mu_3+\mu_4))+\frac{e_6}{28}\otimes(47\mu_1+37\mu_7+3\sqrt{15}\mu_9)+\frac{2e_7}{28}\otimes\mu_8,\\
v_6=&e_1\otimes(-\mu_6+\mu_{10})+\frac{e_2}{28}\otimes(3\sqrt{15}\mu_2+13\mu_3+41\mu_4)+\frac{2e_3}{7}\otimes\mu_8+\frac{e_4}{28}\otimes(3\sqrt{15}(\mu_1-\mu_7)+41\mu_9)\\&+\frac{e_5}{28}\otimes(37\mu_1+47\mu_7-3\sqrt{15}\mu_9)+\frac{e_6}{28}\otimes(-37\mu_2+3\sqrt{15}(\mu_3+\mu_4)),\\
v_7=&\frac{e_1}{14}\otimes(5(\mu_1-\mu_7)+3\sqrt{15}\mu_9)-\frac{e_3}{28}\otimes(3\sqrt{15}(\mu_1-\mu_7)+13\mu_9)+\frac{5e_4}{14}\otimes\mu_8-2e_5\otimes(\mu_6+\mu_{10})\\&+e_6\otimes(-2\mu_5+\mu_{11})-\frac{e_7}{28}\otimes(3\sqrt{15}\mu_2+13\mu_3+41\mu_4).
\end{align*}The subspace of $V_{(0,2)}$ isomorphic to $S^6\C^2$ is $\textup{Span}_\C\{e_i,i=1\ldots7\}$ and the SO(3) equivariant homomorphism  $\alpha$ between $V_{(0,2)}$ and $(\m^*\otimes\lieg2)_\C$ is given by \begin{align*}
    \alpha(e_1)&=v_1, \ \ \ \alpha(e_2)=v_7, \  \  \  \alpha(e_3)=-v_5, \\
    \alpha(e_4)&=-2v_4, \  \ \ \alpha(e_5)=2v_3, \ \ \ \alpha(e_6)=-v_2, \ \ \ \alpha(e_7)=v_6.
\end{align*} Any section of the bundle associated to $\m^*\otimes \lieg2$ in $\ker((D^{-1/3.can})^2-\frac{49}{9}\id)$ can be represented by $(\alpha,v)$ for some $v\in V_{(0,2)}|_{S^6\C^2}\cong \m^*_\C$. The action of the canonical connection on such a section is then given by $\del^{-1,can}_X(\alpha,v)(eH)=-\alpha([X,v])$ where the Lie bracket is in $\mathfrak{so}(5)$. We can now calculate the action of the Dirac operator, $D^{-1,can}$ on $(\alpha,e_1)\cdot\eta$ at the point $eH$ as follows. We omit the $\cdot\eta$ from the computations to reduce notational clutter and will continue to do so in every case.
\begin{align*}
    D^{-1,can}(\alpha,e_1)(eH)&= \sum_{i=1}^7 e_i\cdot\del^{-1,can}_{e_i}(\alpha,e_1)(eH)\\
    &= \frac{-2}{3}(e_2\cdot\alpha(e_4)+e_3\cdot\alpha(e_7)+e_4\cdot\alpha(-e_2)+e_5\cdot\alpha(e_6)+e_6\cdot\alpha(-e_5)+e_7\cdot\alpha(-e_3))\\
    &=\frac{2}{3}(2e_2\cdot v_4-e_3\cdot v_6+e_4\cdot v_7 +e_5\cdot v_2+2e_6\cdot v_3-e_7\cdot v_5)\\
    &=\frac{2}{3}(-3v_1)\cdot\eta = -2\alpha(e_1).
    \end{align*} Thus by the translation invariance of the canonical connection $ V_{(0,2)}\subseteq\ker(D^{-1,can}+2\id)\cap\Gamma(\m^*\otimes\lieg2)_\C$. 

\subsubsection{$\frac{\mathrm{Sp}(2)\times\mathrm{Sp}(1)}{\mathrm{Sp}(1)\times \mathrm{Sp}(1)}$}

From the previous section we know that for $E=\mathfrak{sp}(1)\oplus\mathfrak{sp}(1)$ the  $\ker((D^{-1/3.can})^2-\frac{49}{9}\id)\cap\Gamma(\m^*\otimes E)_\C\cong  V_{(0,1,0)}$. Let  $\{e_i,i=1\dots 7\}$ be an  orthonormal basis of $\m^*$ with respect to the metric $-\frac{3}{40} B$ given by\begin{align*}
e_1 &:= \frac{1}{3}\left(\begin{pmatrix}
0&0\\0&2i
\end{pmatrix},-3i\right), \ \ \ e_2:= \frac{1}{3}\left(\begin{pmatrix}
0&0\\0&2j
\end{pmatrix},-3j\right), \ \ \  e_3:= \frac{1}{3}\left(\begin{pmatrix}
0&0\\0&2k
\end{pmatrix},-3k\right),\\
e_4&:= \frac{\sqrt{5}}{3}\left(\begin{pmatrix}
0&1\\-1&0
\end{pmatrix},0\right), \   e_5:= \frac{\sqrt{5}}{3}\left(\begin{pmatrix}
0&i\\i&0
\end{pmatrix},0\right), \  e_6:= \frac{\sqrt{5}}{3}\left(\begin{pmatrix}
0&j\\j&0
\end{pmatrix},0\right), \   e_7:=\frac{\sqrt{5}}{3} \left(\begin{pmatrix}
0&k\\k&0
\end{pmatrix},0\right).
\end{align*}
With respect to this basis the nearly $\G2$ form is given by\begin{align*}
\g2&=e_{123}-e_{145}-e_{167}-e_{246}+e_{257}-e_{347}-e_{356},
\end{align*}

From Table \ref{tablekerDsquare} we know that as an $\mathrm{Sp}(1)\times\mathrm{Sp}(1)$ connection the deformation space of the canonical connection is an irreducible subrepresentation of  $V_{(0,1,0)}$ and is thus trivial or $(V_{(0,1,0)})_\R$ . We need to check whether this space lies in the $-2$ eigenspace of $D^{-1,A}$

\medskip
\noindent
The $\mathrm{Sp}(2)\times\mathrm{Sp}(1)$-representation $V_{(0,1,0)}$ is 5 dimensional. We need to find the space Hom$(V_{(0,1,0)},(\m^*\otimes (\mathfrak{sp}(1)_u\oplus\mathfrak{sp}(1)_d))_\C)_{\mathrm{Sp}(1)\times\mathrm{Sp}(1)}$. The common irreducible $\mathrm{Sp}(1)\times\mathrm{Sp}(1)$ representations in $V_{(0,1,0)}$ and $(\m^*\otimes\mathfrak{sp}(1)_u)_\C$ is $PQ$. Let $S^2P = Span\{I,J,K\}$ then the subspace of $(\m^*\otimes \mathfrak{sp}(1)_u)_\C$ isomorphic to the space $PQ$ is given by $\textup{Span}_\C\{v_1,v_2,v_3,v_4\}$ where 
\begin{align*}
v_1&=e_5\otimes I+e_6\otimes J+e_7\otimes K, \ \ \ v_2=-e_4\otimes I+e_7\otimes J-e_6\otimes K, \\ v_3&=-e_7\otimes I-e_4\otimes J+e_5\otimes K, \ \ v_4=e_6\otimes I-e_5\otimes J-e_4\otimes K.
\end{align*} Let the subspace of $V_{(0,1,0)}$ isomorphic to $PQ$ be given by Span$\{w_1,w_2,w_3,w_4\}$ and the homomorphism space $\textup{Hom}(V_{(0,1,0)},(\m^*\otimes \mathfrak{sp}(1)_u)_\C) = \textup{Span}(\beta)$ where $\beta$ is defined by\begin{align*}
    w_1&\mapsto v_3+iv_4, \ \ \ w_2\mapsto v_1-iv_2, \\ w_3&\mapsto v_1+iv_2, \ \ \ w_4\mapsto v_3-iv_4.
\end{align*} 
Using this isomprhism one can compute that the only non-trivial $\mathfrak{gl}(V_{(0,1,0)}|_{PQ})$ elements with respect to the basis $\{w_1,w_2,w_3,w_4\}$ are \begin{align*}
   \tau_* (e_1)&= \frac{2}{3}\begin{bmatrix}
    i&0&0&0\\0&-i&0&0\\0&0&i&0\\0&0&0&-i
    \end{bmatrix}, \ \ \ \tau_*(e_2)= \frac{2}{3}\begin{bmatrix}
    0&-1&0&0\\1&0&0&0\\0&0&0&1\\0&0&-1&0
    \end{bmatrix}, \ \ \ \tau_*(e_3)= \frac{2}{3}\begin{bmatrix}
    0&i&0&0\\i&0&0&0\\0&0&0&-i\\0&0&-i&0
    \end{bmatrix}.
\end{align*}Also by the definition of the canonical connection, $\del^{-1,can}_X (\beta,w) (eH)= -\beta(\tau_*(X)w)$. Thus we can calculate \begin{align*}
    (D^{-1,can}(\beta,w_1))(eH)&= \sum_{i=1}^7 e_i\cdot\del^{-1,can}_{e_i}(\beta,w_1)(eH)=-\sum_{i=1}^7e_i\cdot\beta((\tau_*(e_i)w_1)|_{PQ})\\ &=-(e_1\cdot \beta(\frac 23 iw_1)+e_2\cdot \beta(\frac 23 w_2)+e_3\cdot\beta(\frac 23 iw_2))\\
    &= -\frac 23 (ie_1\cdot (v_3+iv_4)+e_2\cdot (v_1-iv_2)+ie_3\cdot(v_1-iv_2))\\
    &=-\frac{2}{3}(3(v_3+iv_4))=-2\beta(w_1).
\end{align*}  Thus we have shown that $V_{(0,1,0)}$ lies in the $\ker(D^{-1,can}+2\id)$.

For $E=\lieg2$ the subspace of $\Gamma(\m^*\otimes\lieg2)$ in $\ker((D^{-1/3.can})^2-\frac{49}{9}\id)$ is isomorphic to the $\mathrm{Sp}(1)\times\mathrm{Sp}(1)$ representation $2V_{(2,0,0)}\oplus V_{(0,1,0)}\oplus V_{(0,0,2)}$. We have already dealt with the space $V_{(0,1,0)}$. The remaining spaces are $2V_{(2,0,0)}\cong 2\mathfrak{sp}(2)$ and $V_{(0,0,2)}\cong \mathfrak{sp}(1)$. The two copies of $V_{(2,0,0)}$ arise from $\textup{Hom}(V_{(2,0,0)},\m_\C^*\otimes PS^3Q)_{\mathrm{Sp}(1)\times\mathrm{Sp}(1)}$ and the one copy of $V_{(0,0,2)}$ arises from $\textup{Hom}(V_{(0,0,2)},\m_\C^*\otimes PS^3Q)_{\mathrm{Sp}(1)\times\mathrm{Sp}(1)}$. Thus we have two cases:

\vspace{0.3cm}
\noindent
\underline{\textbf{Case: 1}-$\textup{Hom}(V_{(0,0,2)},\m_\C^*\otimes PS^3Q)_{\mathrm{Sp}(1)\times\mathrm{Sp}(1)}\otimes V_{(0,0,2)}$}

\noindent
Let $\{w_1,w_2,w_3\}$ be the standard basis of $V_{(0,0,2)}\cong\mathfrak{sp}(1)_\C$ then the non-trivial actions of $\m$ on $\mathfrak{sp}(1)_\C$  are given by 
\begin{align*}
    [e_1,.]&=\begin{bmatrix}
    0&0&0\\0&0&-2\\0&2&0
    \end{bmatrix}, \ \ \ [e_2,.]=\begin{bmatrix}
    0&0&2\\0&0&0\\-2&0&0
    \end{bmatrix}, \ \ \ [e_3,.]=\begin{bmatrix}
    0&2&0\\2&0&0\\0&0&0
    \end{bmatrix}.
\end{align*} 
Let $\{\mu_i,i=1\dots 8\}$ be a basis of the $\textup{Sp}(1)_u\times \textup{Sp(1)}_d$ subrepresentation of $(\mathfrak{g}_2)_\C$ isomorphic to $PS^3Q$. The $1$-dimensional space Hom$(V_{(0,0,2)},(\m^*\otimes \mathfrak{g}_2)_\C) = \textup{Span}\{\phi\}$ where  $\phi$ maps \begin{align*}
    w_1&\mapsto e_4\otimes (\mu_5-\mu_2)+e_5\otimes(\mu_1+\mu_6)+e_6\otimes(\mu_4-\mu_7)-e_7\otimes(\mu_3+\mu_8),\\
    w_2&\mapsto e_4\otimes(\mu_3-2\mu_8)-e_5\otimes(\mu_4+2\mu_7)+e_6\otimes(\mu_1-2\mu_6)-e_7\otimes(\mu_2+2\mu_5),\\
    w_3&\mapsto -e_4\otimes(2\mu_4+\mu_7)+e_5\otimes(\mu_8-2\mu_3)-e_6\otimes(2\mu_2+\mu_5)+e_7\otimes(\mu_6-2\mu_1).
    \end{align*} 
    The connection $\del^{-1,can}_X(\phi,w)= -\phi([X,w])$ for $w\in \mathfrak{sp}(1)$ where the Lie bracket is in the Lie algebra $\mathfrak{sp}(2)\oplus\mathfrak{sp}(1)$. Thus we can calculate\begin{align*}
        D^{-1,can}(\phi,w_1)(eH)&= \sum_{i=1}^7e_i\cdot\del^{-1,can}_{e_i}(\phi,w_1)(eH)=-\sum_{i=1}^7 e_i\cdot \phi([e_i,w_1])\\
        &=-(e_2\cdot \phi(-2w_3)+e_3\cdot \phi(2w_2)) \\
        &=-2(e_4\otimes (\mu_5-\mu_2)+e_5\otimes(\mu_1+\mu_6)+e_6\otimes(\mu_4-\mu_7)-e_7\otimes(\mu_3+\mu_8))\\& = -2\phi(w_1).
    \end{align*} Hence again by translation invariance of $\del^{-1,can}$, $V_{(0,0,2)}\subseteq \ker(D^{-1,can}+2\id)\cap\Gamma(\m^*\otimes\lieg2)_\C.$ 

\medskip

\noindent
\underline{\textbf{Case: 2}-$\textup{Hom}(V_{(2,0,0)},\m_\C^*\otimes PS^3Q)_{\mathrm{Sp}(1)\times\mathrm{Sp}(1)}\otimes V_{(2,0,0)}$}

\noindent
The $\mathrm{Sp}(2)\times\mathrm{Sp}(1)$-representation  $V_{(2,0,0)}\cong \mathfrak{sp}(2)_\C\cong S^2P\oplus S^2Q\oplus PQ$. The subspace of $(\mathfrak{sp}(2))_\C$ isomorphic to $S^2Q, PQ$ is given by $\textup{Span}_\C\{e_1,e_2,e_3\}, \textup{Span}_\C\{e_4,e_5,e_6,e_7\}$ respectively. As before the basis of $PS^3Q\subset (\fg_2)_\C$ is denoted by $\{\mu_1,\mu_2,\dots,\mu_8\}$ and the subspace of $(\m^*\otimes \fg_2)_\C$ isomorphic to $S^2Q$ is given by $\textup{Span}\{\phi(w_1),\phi(w_2),\phi(w_3)\}$ defined above.
The subspace of $(\m^*\otimes \fg_2)_\C$ isomorphic to $PQ$ is given by $\textup{Span}\{v_1,v_2,v_3,v_4\}$ where
\begin{align*}
v_1&=e_1\otimes(\mu_1+\mu_6)-e_2\otimes (\mu_4+2\mu_7)-e_3\otimes(2\mu_3-\mu_8),\\
v_2&=e_1\otimes(\mu_2-\mu_5)-e_2\otimes (\mu_3-2\mu_8)+e_3\otimes(2\mu_4+\mu_7),\\
v_3&=-e_1\otimes(\mu_3+\mu_8)-e_2\otimes (\mu_2+2\mu_5)-e_3\otimes(2\mu_1-\mu_6),\\
v_4&=-e_1\otimes(\mu_4-\mu_7)-e_2\otimes (\mu_1-2\mu_6)+e_3\otimes(2\mu_2+\mu_5).
    \end{align*}  
Let $\{A_1,A_2\}$ be a basis of the $2$-dimensional space $\textup{Hom}(V_{(2,0,0)},(\m^*\otimes\fg_2)_\C)_{\mathrm{Sp}(1)_u\times \mathrm{Sp}(1)_d}$ and let $A=c_1A_1+c_2A_2$ for some real constants $c_1,c_2$ then we have that
\begin{align*}
A(e_1)&=c_1w_1, \ \ A(e_2)=c_1w_2, \ \ A(e_3)=c_1w_3\\
  A(e_4)&=-c_2v_2, \ \  A(e_5)=c_2v_1,\ \
  A(e_6)=-c_2v_4,\ \ 
  A(e_7)=c_2v_3
\end{align*}
and $A_1,A_2$ acts trivially on $S^2P$.

Let $s_{(A,w)}\in\Gamma(\m^*\otimes \fg_2)_\C$ be the section corresponding to $(A,w)\in \textup{Hom}(V_{(2,0,0)},(\m^*\otimes \fg_2)_\C)_{\mathrm{Sp}(1)\times\mathrm{Sp}(1)}\otimes \mathfrak{sp}(2)$ then $\del^{-1,can}_X(A,w)=-A(\operatorname{ad}(X)w)= A([X,w]|)$ where the Lie bracket is in the Lie algebra $\mathfrak{sp}(2)$. Using this action of $\del^{-1,can}$ we can calculate
\begin{align*}
(D^{-1,can}(A,e_1))(eH)&= \sum_{i=1}^7e_i\cdot \del^{-1,can}_{e_i}(A,e_1)(eH) = -\sum_{i=1}^7e_i\cdot A([e_i,e_1]|)\\
&= -\frac 23(-e_2\cdot A(e_3)+e_3\cdot A(e_2)+e_4\cdot A(e_5)-e_5\cdot A(e_4)+e_6\cdot A(e_7)-e_7\cdot A(6))\\
&= -\frac 23(c_1(-e_2\cdot w_3+e_3\cdot w_2)+c_2(e_4\cdot v_1 -e_5\cdot (-v_2)+e_6\cdot v_3-e_7\cdot (-v_4)))\\
&=\frac{4c_1-6c_2}{3}w_1=\frac{4c_1-6c_2}{3}A_1(e_1).
\end{align*}
By doing similar computations we get that
\begin{align*}
(D^{-1,can}(A,f_i))(eH)&=0, \ \ \ i=1,2,3,\\
(D^{-1,can}(A,e_i))(eH)&= \frac{4c_1-6c_2}{3}A_1(e_i), \ \ \ i=1,2,3,\\
   (D^{-1,can}(A,e_i))(eH)&= -\frac{20c_1+6c_2}{9}A_2(e_i), \ \ \ i=4,5,6,7.
\end{align*}

Therefore the subspace of $\textup{Hom}(V_{(2,0,0)},(\m^*\otimes \fg_2)_\C)_{\mathrm{Sp}(1)\times\mathrm{Sp}(1)}$ in the $\ker(D^{-1,can}+2\id)$ is given by the condition $c_2=\frac{5}{3}c_1$ and is thus $1$-dimensional. Therefore $V_{(2,0,0)}$ occurs in the $\ker(D^{-1,can}+2\id)\cap\Gamma(\m^*\otimes\fg_2)_\C$ with multiplicity $1$. 

\begin{rem}\label{eigenvalueeigthby3}
We can immediately see from above that the only other possible eigenvalue for which $\mathfrak{sp}(2)$ is an eigenspace of $D^{-1,can}$ is $-\frac{8}{3}$ for $c_2=-\frac{2}{3}c_1$. This shows that not all spaces in $\ker((D^{-1/3,can})^2-\frac{49}{9}\id)$ are in $\ker(D^{-1,can}+2\id)$.
\end{rem}
\subsubsection{ $\frac{\mathrm{SU}(3)\times\mathrm{SU}(2)}{\mathrm{SU}(2)\times \mathrm{U}(1)}$}
As before let $\{e_i,i=1\ldots7\}$ be an orthonormal basis of $\m^*$ with respect to $g$. If we define  $I=\begin{pmatrix}i&0\\0&-i\end{pmatrix}, J=\begin{pmatrix}0&-1\\1&0\end{pmatrix}, K=\begin{pmatrix}0&i\\i&0\end{pmatrix}$ we have

\begin{align*}
  e_1 &:= \frac{1}{3}\left(\begin{pmatrix}
2I&0\\0&0
\end{pmatrix},-3I\right), \ \ \ e_2:= \frac{1}{3}\left(\begin{pmatrix}
2J&0\\0&0
\end{pmatrix},-3J\right), \ \ \ e_3:= \frac{1}{3}\left(\begin{pmatrix}
2K&0\\0&0
\end{pmatrix},-3K\right),
\end{align*}
\begin{align*}
e_4&:= \frac{\sqrt{5}}{3}\left(\begin{pmatrix}
0&0&\sqrt{2}\\0&0&0\\-\sqrt{2}&0&0
\end{pmatrix},0\right), \ \ \ e_5:= \frac{\sqrt{5}}{3}\left(\begin{pmatrix}
0&0&\sqrt{2}i\\0&0&0\\\sqrt{2}i&0&0
\end{pmatrix},0\right), \\ e_6&:= \frac{\sqrt{5}}{3}\left(\begin{pmatrix}
0&0&0\\0&0&\sqrt{2}\\0&-\sqrt{2}&0
\end{pmatrix},0\right), \ \ \ e_7:=\frac{\sqrt{5}}{3} \left(\begin{pmatrix}
0&0&0\\0&0&\sqrt{2}i\\0&\sqrt{2}i&0
\end{pmatrix},0\right).  
\end{align*} 
With respect to this basis the nearly $\G2$ structure $\g2$ is given by \begin{align*}
 \g2&=e_{123}+e_{145}-e_{167}+e_{246}+e_{257}+e_{347}-e_{356}. 
\end{align*} As an $\mathrm{SU}(2)\times \mathrm{U}(1)$ representation, $\m^*_\C\cong S^2W\oplus WF(3)\oplus WF(-3)$ where \begin{align*}
    S^2W &= \textup{Span}\{e_1,e_2,e_3\},\ \ \ 
    WF(3)=\textup{Span}\{e_4-ie_5,e_6-ie_7\},\ \ \ WF(-3)=\textup{Span}\{e_4+ie_5,e_6+ie_7\}.
\end{align*} From our previous work we know that the canonical connection has no deformations as an $\mathrm{SU}(2)\times\mathrm{U}(1)$ connection so we only have to consider the case $E=\lieg2$.

As an $\mathrm{SU}(2)\times \mathrm{U}(1)$ representation, $(\mathfrak{g}_2)_\C\cong S^3W(F(3)\oplus F(-3))\oplus S^2W\oplus F(6)\oplus F(-6)$. 
We have already seen that $S^2W$ gives rise to no deformations. From previous calculations we know that $\ker((D^{-1/3,can})^2-\frac{49}{9}\id)\cap\Gamma(\m_\C^*\otimes S^3WF(\pm 3))\cong V_{(0,0,2)}\oplus 2V_{(1,1,0)}\cong (\mathfrak{su}(2))_\C\oplus 2(\mathfrak{su}(3))_\C$ and $\Gamma(\m_\C^*\otimes F(\pm 6))\cap \ker((D^{-1/3,can})^2-\frac{49}{9}\id)\cong V_{(1,1,0)}$ respectively.
Therefore there are $6$ subspaces of $\Gamma(\m^*\otimes\lieg2)$ to consider here.
\medskip

\noindent
\underline{\textbf{Case: 1}-$\textup{Hom}(V_{(0,0,2)},\m_\C^*\otimes S^3WF(3))_{\mathrm{SU(2)\times U(1)}}\otimes V_{(0,0,2)}$}

\medskip
\noindent
We denote by $\{\mu_i,i=1\ldots4\}$ a basis of $S^3WF(3)$. Let $f_i,i=1\ldots3$ be the standard basis of $\mathfrak{su}(2)$ such that $[f_1,f_2]=-2f_3, [f_1,f_3]=2f_2, [f_2,f_3]=-2f_1$. Then the subspace of $WF(-3)\otimes S^3WF(3)\subset(\m^*\otimes\lieg2)_\C$ isomorphic to $(\mathfrak{su}(2))_\C$ is given by $\textup{Span}\{v_1,v_2,v_3\}$ where \begin{align*}
    v_1&= \frac{3i}{4}(e_4+ie_5)\otimes \mu_1 + (e_6+ie_7)\otimes (\frac{5i}{4}\mu_2+\mu_4),\\
    v_2&=(e_4+ie_5)\otimes (-i\mu_2+\mu_4)+(e_6+ie_7)\otimes(-i\mu_1-\mu_3),\\
    v_3&= (e_4+ie_5)\otimes (-\frac{5i}{4}\mu_1+\mu_3) -\frac{3i}{4} (e_6+ie_7)\otimes\mu_2
\end{align*}  and the space $\textup{Hom}(V_{(0,0,2)}, (\m^*\otimes\fg_2)_{\C}) = \textup{Span}\{\gamma^A\}$ where $\gamma^A$ is defined by
\begin{align*}
    \gamma^A(f_1)&=v_2, \ \ \ \gamma^A(f_2)=i(v_1-v_3),\ \ \ \gamma^A(f_3)=-2(v_1+v_3).
\end{align*}
For $i=1,2,3,$ since $e_i=(\frac{2}{3}f_i,-f_i)$ we have $[e_i,v]=-[f_i,v]$ for all $v\in\mathfrak{su}(2)$. The action is trivial for $i=4\dots 7$ since $[e_i,f_j]\notin\textup{Span}\{f_1,f_2,f_3\}$. We can thus calculate 
\begin{align*}
    D^{-1,can}(\gamma^A,f_1)(eH)&= \sum_{i=1}^7 e_i\cdot \del^{-1,can}_{e_i}(\gamma^A,f_1)(eH)\\ 
    &=e_2\cdot \gamma^A(2f_3)-e_3\cdot \gamma^A(2f_2) \\
    &=-(4e_2\cdot (v_1+v_3) +2i e_3\cdot (v_1-v_3)) \\
    &=-2v_2 = -2 \ \gamma^A(f_1).
\end{align*} Hence $\textup{Hom}(V_{(0,0,2)},\m_\C^*\otimes S^3WF(3))|_{\mathrm{Sp(1)\times Sp(1)}}\otimes V_{(0,0,2)}\subseteq \ker(D^{-1,can}+2\id)$.
\vspace{0.5cm}

\noindent
\underline{\textbf{Case: 2}-$\textup{Hom}(V_{(1,1,0)},\m_\C^*\otimes S^3WF(3))_{\mathrm{SU(2)\times U(1)}}\otimes V_{(1,1,0)}$}

\noindent
Let a basis of the subspace of $ V_{(1,1,0)}\cong(\mathfrak{su}(3))_\C$ isomorphic to $S^2W\cong (\mathfrak{su}(2))_\C$ be given by 
\begin{align*}
    p_1&:=\begin{pmatrix}I&0\\0&0\end{pmatrix}, \ \ \ p_2:=\begin{pmatrix}J&0\\0&0\end{pmatrix}, \ \ \ p_3:=\begin{pmatrix}K&0\\0&0\end{pmatrix}.
\end{align*} 
where $I,J,K$ are defined previously. Then $[p_1,p_2]=-2p_3,[p_1,p_3]=2p_2, [p_2,p_3]=-2p_1$. The basis of $\m_\C^*\otimes S^3WF(3)\subset \m_\C^*\otimes  \lieg2$ isomorphic to $S^2W$ is given by $\textup{Span}\{w_1,w_2,w_3\}$ where
\begin{align*}
    w_1&=(e_4+ie_5)\otimes\frac{\mu_2+i\mu_3}{2}+(e_6+ie_7)\otimes\frac{\mu_1-i\mu_4}{2},\\
    w_2&=(e_4+ie_5)\otimes\frac{\mu_4-2i\mu_1}{2}+(e_6+ie_7)\otimes\frac{\mu_3-2i\mu_2}{2},\\
    w_3&=-(e_4+ie_5)\otimes\frac{\mu_1+2i\mu_4}{2}+(e_6+ie_7)\otimes\frac{\mu_2-2i\mu_3}{2}.
\end{align*}

Since $(\mathfrak{su}(3))_\C=\m_\C\oplus\C$, the subspace of $(\mathfrak{su}(3))_\C$ isomorphic to $WF(3)$ is given by $\textup{Span}_\C\{e_4-ie_5,e_6-ie_7\}$. The subspace of $S^2W\otimes S^3WF(3)\subset(\m^*\otimes \lieg2)_\C$ isomorphic to $WF(3)$ is given by $\textup{Span}\{u_1,u_2\}$ where 
\begin{align*}
    u_1&=ie_1\otimes\frac{\mu_2+i\mu_3}{2}+e_2\otimes\frac{2\mu_1+i\mu_4}{2}-ie_3\otimes\frac{\mu_1+2i\mu_4}{2},\\
    u_2&=ie_1\otimes\frac{\mu_1-i\mu_4}{2}+e_2\otimes\frac{2\mu_2-i\mu_3}{2}+ie_3\otimes\frac{\mu_2-2i\mu_3}{2}
\end{align*}
 If we denote the space $\textup{Hom}(V^{(1,1,0)},\m_\C^*\otimes S^3WF(3))$ and $\textup{Hom}(V^{(1,1,0)},\m_\C^*\otimes S^3WF(3))$  by $\textup{Span}\{A_1\}$, $\textup{Span}\{A_2\}$ respectively then \begin{align*}
    A_1(p_i)&=w_i, \ \ \ i=1,2,3,\\
    A_2(e_4-ie_5)&=u_1, \ \ A_2(e_6-ie_7)=u_2.
\end{align*}

Define $A=c_1A_1+c_2A_2$ for some constants $c_1,c_2$. We need to find the conditions on $c_1,c_2$ such that $(A,w)\in\Gamma(\m^*\otimes S^3WF(3))\cap\ker(D^{-1,can}+2\id)$ for all $w\in\mathfrak{su}(3)$. 

Let $s_{(A,w)}$ be the section corresponding to $(A,w)$. Then for any vector field $X$, $\del^{-1,can}_X(A,w)=-A(\operatorname{ad}(X)w)= A([X,w]|)$ where the Lie bracket is in the Lie algebra $\mathfrak{su}(3)$. Using this action of $\del^{-1,can}$ we can calculate 
\begin{align*}
    D^{-1,can}(A,p_1)(eH)&= \sum_{i=1}^7 e_i\cdot \del^{-1,can}_{e_i}(A,p_1)(eH)\\ 
    &= -(\frac{2}{3}(-e_2\cdot A(2p_3)+e_3\cdot A(2p_2))e_4\cdot A(-e_5)+e_5\cdot A(e_4)+e_6\cdot A(e_7)+e_7\cdot A(e_6)) \\
    &=-\frac{2c_1}{3}(-e_2\cdot w_1 + e_3\cdot w_2)-c_2(-e_4\cdot i\frac{u_1}{2}+e_5\cdot \frac{u_1}{2}+e_6\cdot i\frac{u_2}{2} -e_7\cdot\frac{u_2}{2}) \\
    &=\frac{4c_1+3ic_2}{3}w_1=\frac{4c_1+3ic_2}{3}A_1(e_1).
\end{align*}

The operator $D^{-1,can}$ acts trivially on the subspaces of $(\mathfrak{su}(3))_\C$ isomorphic to $\C$ and $WF(-3)$. On the remaining subspaces we can compute the action of the Dirac operator as \begin{align*}
   D^{-1,can}(A,p_1)(eH)&= \frac{4c_1+3ic_2}{3}A_1(e_i), \ \ \ i=1,2,3,\\
    D^{-1,can}(A,e_4-ie_5)(eH)&= \frac{20c_1-3ic_2}{9}A_2(e_4-ie_5),\\
     D^{-1,can}(A,e_6-ie_7)(eH)&= \frac{20c_1-3ic_2}{9}A_2(e_6-ie_7).
\end{align*} 
Thus for any $w\in(\mathfrak{su}(3))_\C$, $(A,w)\in\ker(D^{-1,can}+2\id)$ if and only if $c_2=\frac{10i}{3}c_1$. Thus only one copy of $\mathfrak{su}(3)$ lies in $\ker(D^{-1,can}+2\id)$.

Note that similarly to Remark \ref{eigenvalueeigthby3} here also for $c_2=-\frac{4i}{3}c_1$, $(A,w)\in\ker(D^{-1,can}-\frac{8}{3}\id)$. 
\medskip

\noindent
\underline{\textbf{Case: 3}-$\textup{Hom}(V_{(0,0,2)},\m_\C^*\otimes S^3WF(-3))_{\mathrm{Sp(1)\times Sp(1)}}\otimes V_{(0,0,2)}$}

\medskip
\noindent
Let $f_i,i=1\ldots3$ be as before and denote by $\{\nu_i,i=1\ldots4\}$ a basis of $S^3WF(-3)$. Then the subspace of $WF(3)\otimes S^3WF(-3)$ isomorphic to $S^2W$ is given by $\textup{Span}\{w_1,w_2,w_3\}$ where \begin{align*}
    w_1&= (e_4-ie_5)\otimes (\frac{-3i}{4}\nu_1 )+ (e_6-ie_7)\otimes (\frac{-5i}{4}\nu_2+\nu_4),\\
    w_2&=(e_4-ie_5)\otimes (i\nu_2+\nu_4)+(e_6-ie_7)\otimes(i\nu_1-\nu_3),\\
    w_3&= (e_4-ie_5)\otimes (\frac{5i}{4}\nu_1+\nu_3) + (e_6-ie_7)\otimes(\frac{3i}{4}\nu_2)
\end{align*}  and the space $\textup{Hom}(V_{(0,0,2)}, (\m_C^*\otimes S^3WF(-3))) = \textup{Span}\{\gamma^B\}$ where $\gamma^B$ is defined by \begin{align*}
    \gamma^B(f_1)&=\frac{i}{2}w_2, \ \ \ \gamma^B(f_2)=\frac{1}{2}(w_1-w_3),\ \ \ \gamma^B(f_3)=-i(w_1+w_3).
\end{align*} The action of $e_i,i=1\ldots7$ on $f_j,j=1\ldots3$ is the same as Case 1 and thus we can calculate $D^{-1,can}(\gamma^B,f_1)$ as \begin{align*}
    D^{-1,can}(\gamma^B,f_1)(eH)&= \sum_{i=1}^7 e_i\cdot \del^{-1,can}_{e_i}(\gamma^B,f_1)(eH)\\ 
    &= e_2\cdot \gamma^B(2f_3)-e_3\cdot \gamma^B(2f_2) \\
    &=-2ie_2\cdot (w_1+w_3) - e_3\cdot (w_1-w_3) \\
    &=-iw_2=-2\ \gamma^B(f_1).
    \end{align*}
    This implies $V_{(0,0,2)}\subseteq \ker(D^{-1,can}+2\id)\cap\Gamma(\m^*\otimes\fg_2)_\C$.
\medskip

\noindent
\underline{\textbf{Case: 4}-$\textup{Hom}(V_{(1,1,0)},\m_\C^*\otimes S^3WF(-3))_{\mathrm{SU(2)\times U(1)}}\otimes V_{(1,1,0)}$}

\medskip
\noindent
As above in Case 2, let a basis of the subspace of $(\mathfrak{su}(3))_\C$ isomorphic to $S^2W\cong \mathfrak{su}(2)$ be given by $\textup{Span}\{p_1,p_2,p_3\}$.  The basis of $\m_\C^*\otimes S^3WF(-3)\subset (\m^*\otimes  \lieg2)_\C$ isomorphic to $S^2W$ is given by $\textup{Span}\{w_1,w_2,w_3\}$ where 
\begin{align*}
   w_1&=(e_4-ie_5)\otimes\frac{\nu_2-i\nu_3}{2}+(e_6-ie_7)\otimes\frac{\nu_1+i\nu_4}{2},\\
    w_2&=(e_4-ie_5)\otimes\frac{\nu_4+2i\nu_1}{2}+(e_6-ie_7)\otimes\frac{\nu_3+2i\nu_2}{2},\\
    w_3&=-(e_4-ie_5)\otimes\frac{\nu_1-2i\nu_4}{2}+(e_6-ie_7)\otimes\frac{\nu_2+2i\nu_3}{2}. 
\end{align*}
The subspace of $(\mathfrak{su}(3))_\C$ isomorphic to $WF(-3)$ is given by $\textup{Span}\{e_4+ie_5,e_6+ie_7\}$. The subspace of $S^2W\otimes S^3WF(-3)\subset(\m^*\otimes \lieg2)_\C$ isomorphic to $WF(-3)$ is given by $\textup{Span}_\C\{u_1,u_2\}$ where 

\begin{align*}
    u_1&=-ie_1\otimes\frac{\nu_2-i\nu_3}{2}+e_2\otimes\frac{2\nu_1-i\nu_4}{2}+ie_3\otimes\frac{\nu_1-2i\nu_4}{2},\\
    u_2&=-ie_1\otimes\frac{\nu_1+i\nu_4}{2}+e_2\otimes\frac{2\nu_2+i\nu_3}{2}-ie_3\otimes\frac{\nu_2+2i\nu_3}{2}.
\end{align*}
Again if we denote the spaces $\textup{Hom}(V{(1,1,0)},\m_\C^*\otimes S^3WF(-3))$ and $\textup{Hom}(V^{(1,1,0)},\m_\C^*\otimes S^3WF(-3))$ by $\textup{Span}\{B_1\}$, $\textup{Span}\{B_2\}$ respectively then \begin{align*}
    B_1(p_i)&=w_i, \ \ \ i=1,2,3,\\
    B_2(e_4+ie_5)&=u_1, \ \ B_2(e_6+ie_7)=u_2.
\end{align*}
Again as before we need to find the conditions on $c_1,c_2$ such that $(B=c_1B_1+c_2B_2,w)\in\ker(D^{-1,can}+2\id)$ for all $w\in(\mathfrak{su}(3))_\C$. 
By similar computations as Case 2, we can calculate,
\begin{align*}
   D^{-1,can}(B,p_1)(eH)&= \sum_{i=1}^7 e_i\cdot \del^{-1,can}_{e_i}(B,p_1)(eH)\\ 
    &= -(\frac{2}{3}(-e_2\cdot B(2p_3)+e_3\cdot B(2p_2))+e_4\cdot B(-e_5)+e_5\cdot B(e_4)+e_6\cdot B(e_7)+e_7\cdot B(e_6)) \\
    &=-\frac{2c_1}{3}(-e_2\cdot w_1 + e_3\cdot w_2)-c_2(-e_4\cdot i\frac{u_1}{2}+e_5\cdot \frac{u_1}{2}+e_6\cdot i\frac{u_2}{2} -e_7\cdot\frac{u_2}{2}) \\
    &=\frac{4c_1-3ic_2}{3}w_1=\frac{4c_1-3ic_2}{3}B_1(e_1).
\end{align*}
Once can check that $D^{-1,can}$ acts trivially on the subspaces of $(\mathfrak{su}(3))_\C$ isomorphic to $\C, WF(3)$ and 
\begin{align*}
   D^{-1,can}(A,p_1)(eH)&= \frac{4c_1-3ic_2}{3}B_1(e_i), \ \ \ i=1,2,3,\\
    D^{-1,can}(A,e_4+ie_5)(eH)&= \frac{20c_1+3ic_2}{9}B_2(e_4+ie_5),\\
     D^{-1,can}(A,e_6+ie_7)(eH)&= \frac{20c_1+3ic_2}{9}B_2(e_6+ie_7).
\end{align*} 
Thus for all $w\in(\mathfrak{su}(3))_\C$, $(B,w)\in\ker(D^{-1,can}+2\id)$ if and only if $c_2=-\frac{10i}{3}c_1$ which proves that only one copy of $\mathfrak{su}(3)$ lies in $\ker(D^{-1,can}+2\id)$ in this case as well. It immediately follows from the given action that for $c_2=\frac{4i}{3}c_1$, $(B,w)\in\ker(D^{-1,can}-\frac{8}{3}\id)$.
\medskip

\noindent
\underline{\textbf{Case: 5}-$\textup{Hom}(V_{(1,1,0)},\m_\C^*\otimes F(6))_{\mathrm{SU(2)\times U(1)}}\otimes V_{(1,1,0)}$} 
    
    \medskip
    \noindent
    From before we know that the subspace of $(\mathfrak{su}(3))_\C$ isomorphic to $WF(3)$ is given by $\textup{Span}\{e_4-ie_5,e_6-ie_7\}$. if we denote by $\mu$ a basis vector for the $1$-dimensional representation $F(6)$, the subspace of $\m_\C^*\otimes F(6)$ isomorphic to $WF(3)$ is given by $\textup{Span}_\C\{(e_4+ie_5)\otimes \mu, (e_6+ie_7)\otimes\mu\}$. Let Hom$(V_{(1,1,0)},\m_\C^*\otimes F(6))=\textup{Span}\{\alpha\}$. 
    We can define $\alpha$ as follows,
    \begin{align*}
        \alpha(e_4-ie_5)&=(e_6+ie_7)\otimes \mu , \ \ \  \alpha(e_6-ie_7)=-(e_4+ie_5)\otimes \mu.
    \end{align*} 
    Since $V_{(1,1,0)}$ is isomorphic to the adjoint representation $(\mathfrak{su}(3))_\C$, $\del^{-1,can}_X(\alpha,v)(eH)= -\alpha([X,v])$ where $X\in\m$, $v\in WF(3)\subset \mathfrak{su}(3)$. Thus we can compute \begin{align*}
        D^{-1,can}(\alpha,e_4-ie_5)(eH)&= \sum_{i=1}^7 e_i\cdot \del^{-1,can}_{e_i}(\alpha,e_4-ie_5)(eH)\\ 
    &= -(e_1\cdot \alpha(\frac{2i}{3}(e_4-ie_5))+e_2\cdot \alpha(\frac{2}{3}(e_6-ie_7))+e_3\cdot\alpha(\frac{2i}{3}(e_6-ie_7))) \\
    &=-\frac 23 (ie_1\cdot(e_6+ie_7)\otimes \mu-e_2\cdot (e_4+ie_5)\otimes \mu-ie_3\cdot  (e_4+ie_5)\otimes \mu) \\
    &= -2(e_6+ie_7)\otimes\mu= -2\alpha(e_4-ie_5). 
    \end{align*}
 Therefore $\textup{Hom}(V_{(1,1,0)},\m_\C^*\otimes F(6))_{\mathrm{SU(2)\times U(1)}}\otimes V_{(1,1,0)}\subset \ker(D^{-1,can}+2\id)$ and thus lies in the deformation space.
\vspace{0.3cm}
 
\noindent
    \underline{ \textbf{Case: 6}-$\textup{Hom}(V_{(1,1,0)},\m_\C^*\otimes F(-6))_{\mathrm{SU(2)\times U(1)}}\otimes V_{(1,1,0)}$ } 
    
    \medskip
    \noindent
    The subspace of $(\mathfrak{su}(3))_\C$ isomorphic to $WF(-3)$ is given by $\textup{Span}_\C\{e_4+ie_5,e_6+ie_7\}$. We denote  $F(-6)=\textup{Span}\{\nu\}$. Then $\m_\C^*\otimes F(-6)$ isomorphic to $WF(-3)$ is given by $\textup{Span}\{(e_4-ie_5)\otimes \nu, (e_6-ie_7)\otimes\nu\}$. Let Hom$(V_{(1,1,0)},\m_\C^*\otimes F(-6))=\textup{Span}\{\beta\}$ then \begin{align*}
        \beta(e_4+ie_5)&=-(e_6-ie_7)\otimes \nu , \ \ \  \beta(e_6+ie_7)=(e_4-ie_5)\otimes \nu.
    \end{align*} Since $V_{(1,1,0)}\cong (\mathfrak{su}(3))_\C$, $\del^{-1,can}_X(\beta,v)(eH)= -\beta([X,v])$ where $X\in\m$, $v\in WF(-3)\subset (\mathfrak{su}(3))_\C$. Thus we can compute \begin{align*}
        D^{-1,can}(\beta,e_4+ie_5)(eH)&= \sum_{i=1}^7 e_i\cdot \del^{-1,can}_{e_i}(\beta,e_4+ie_5)(eH)\\ 
    &= -(e_1\cdot \beta(\frac{-2i}{3}(e_4+ie_5))+e_2\cdot \beta(\frac{2}{3}(e_6+ie_7))+e_3\cdot\beta(\frac{-2i}{3}(e_6+ie_7))) \\
    &=-\frac 23 (ie_1\cdot(e_6-ie_7)\otimes \nu+e_2\cdot (e_4-ie_5)\otimes \nu-ie_3\cdot  (e_4-ie_5)\otimes \nu) \\
    &= 2(e_6-ie_7)\otimes\nu= -2\beta(e_4+ie_5),
    \end{align*}
which by translation invariance of $D^{-1,can}$ shows that $\textup{Hom}(V_{(1,1,0)},\m_\C^*\otimes F(-6))_{\mathrm{SU(2)\times U(1)}}\otimes V_{(1,1,0)}\subset \ker(D^{-1,can}+2\id)$.

\medskip

\noindent
\textbf{Summary of the results}: For three out of the four considered normal homogeneous spaces the canonical connection is rigid as an $H$-connection. As a $\G2$-connection the canonical connection has a non-trivial infinitesimal deformation space except for the round $S^7$. Summing up all the results found above we get the following theorem.

\begin{thm}\label{thm:deformspace}
    The infinitesimal deformation space for the canonical connection on the four normal homogeneous nearly $\G2$  spaces $G/H$ when the structure group is  $H$ or $\G2$ is isomorphic to 
    \bgroup
\def\arraystretch{1}
    \begin{table}[H]\label{table2}
\centering
\begin{tabular}{c|c|c}
\hline
\raisebox{-2pt}{$G/H$}  & \multicolumn{2}{c}{Structure group}\\ \cline{2-3}

    &$H$& $\G2$\\
\hline 
$\mathrm{Spin}(7)/\G2$& $0$ & $0$\\ &&\\
$\mathrm{SO}(5)/\mathrm{SO}(3)$ &$0$ & $\mathfrak{so}(5)$\\ &&\\
$\cfrac{\mathrm{Sp}(2)\times \mathrm{Sp}(1)}{\mathrm{Sp}(1)\times \mathrm{Sp}(1)}$ & $ V^{(0,1)}_\R$ & $\mathfrak{sp}(2)\oplus\mathfrak{sp}(1)\oplus V^{(0,1)}_\R$\\ &&\\
$\cfrac{\mathrm{SU}(3)\times \mathrm{SU}(2)}{\mathrm{SU}(2)\times \mathrm{ U}(1)}$&$0$& $2\mathfrak{su}(2) \oplus 4\mathfrak{su}(3)$\\
\hline
\end{tabular}
\label{tabledeformationspace}
\end{table}
\egroup
where $V^{(0,1)}$ is the unique $5$-dimensional complex irreducible $\textup{Sp(2)}$-representation.
\end{thm}

\subsection{Integrability of the deformation spaces}
We now describe some of the deformation spaces obtained in Theorem \ref{thm:deformspace}.

Let $M$ be a nearly $\G2$ manifold. We first observe that for the structure group $\G2$ the space of non-trivial deformations in Theorem \ref{thm:deformspace} are either isomorphic to or contains as a subrepresentation one or multiple copies of the Lie algebra $\fg$ of the automorphism group $G$. A vector field $X$ on $M$ preserves the $\G2$-structure $\g2$ if $\mathcal{L}_X\g2=0$. We denote by $\mathcal{X}$ the space of vector fields on $M$ preserving the $\G2$-structure.
Since the $\G2$-structure on $G/H$ is $G$ invariant, the space $\fg$ is contained in $\mathcal{X}$. Note that if $X\in\mathcal{X}$ then $\mathcal{L}_X\psi=\mathcal{L}_Xg=0$.

Given a parallel section in $\Gamma(\mathfrak{\fg_2}(T^* M) \otimes \textup{Ad}_\P) \subset \Gamma(\Lambda^2T^* M\otimes \textup{Ad}_\P)$, one can define an operator that associates to each vector field in $\mathcal{X}$ an infinitesimal deformation of a $\G2$ instanton on $M$. Such an operator was defined in \cite{bendef} where a similar situation arises when one computes the deformation space of the canonical connection on the homogeneous $6$-dimensional nearly K\"ahler manifolds. 

The next proposition asserts that if we fix a section $\xi\in \Gamma(\mathfrak{\fg_2}(T^* M) \otimes \textup{Ad}_\P) \subset \Gamma(\Lambda^2T^* M\otimes \textup{Ad}_\P)$, then for any vector field $X\in\mathcal{X}$ on $M$ the $\textup{Ad}_\P$ valued $1$-form $\epsilon_X=i_X\xi\in \Gamma(T^*M\otimes\textup{Ad}_\P)$ defines an infinitesimal deformation of the nearly $\G2$ instanton $A$ in the sense of  \eqref{perturbation}. The proof of the proposition follows verbatim from the proof of \cite[Proposition 9]{bendef} and is hence omitted. 

\begin{prop}\label{prop:Liegdeformations}
Let $A$ be an instanton on a principal $G$-bundle $\P$ over a nearly $\G2$ manifold $M$. Let $\xi\in \Gamma(\mathfrak{\fg_2}(T^* M) \otimes \textup{Ad}_\P) \subset \Gamma(\Lambda^2T^* M\otimes \textup{Ad}_\P)$ such that $\del^{-1,A}\xi=0$. Then for any $X\in\mathcal{X}$, $\epsilon_X=i_X\xi\in \Gamma(T^*M\otimes\textup{Ad}_\P)$ satisfies the linearised instanton condition 
\begin{align*}
    d^A\epsilon_X\cdot\eta = 0.
\end{align*}
\end{prop}

The above proposition implies that for each $\xi\in \Gamma(\fg_2(T^*M)\otimes\mathrm{Ad}_\P)$ such that $\del^{-1,A}\xi=0$, there is a copy of $\fg$ in the deformation space of $A$. Thus the multiplicity of $\fg$ in the deformation space can be found by identifying the parallel sections of $\fg_2(T^*M)\otimes \mathrm{Ad}_\P$. On $G/H$, when we see $\P$ as a $\G2$-bundle, every parallel section of $\fg_2(T^*M)\otimes \mathrm{Ad}_\P$ corresponds to an $H$-invariant element of the $H$-representation $\fg_2\otimes\fg_2$ (since $\mathrm{Ad}_\P\cong \fg_2$) and vice-versa. The number of linearly independent $H$-invariant elements of $\fg_2\otimes\fg_2$ is equal to the multiplicity of the trivial $H$-representation in $\fg_2\otimes\fg_2$.

Observe that since $A$ is a $\G2$ instanton, the curvature $F_A\in \Gamma(\fg_2(T^*M)\otimes\mathrm{Ad}_\P)$. When $A=\del^{can}$ is the canonical connection on $G/H$ and $F$ is the curvature, $\del^{-1,can}F=0$ since $\mathrm{Hol}(\del^{can})\subseteq\G2$. Hence by Proposition \ref{prop:Liegdeformations} for every $X\in\mathcal{X}$, $\epsilon_X=i_XF$ defines an infinitesimal deformation of $A=\del^{can}$. 
Using the Bianchi identity and the definition of $\epsilon_X$ we have that 
\begin{align*}
    d^A\epsilon_X&= d\epsilon_X + [A,\epsilon_X] \\
    &= \mathcal{L}_XF - i_XdF +[A,\epsilon_X] \\
    &=\mathcal{L}_XF +i_X[A,F] +[A,\epsilon_X]\\
    &=\mathcal{L}_XF +[i_XA,F].
\end{align*}

Since under the action of a gauge transformation $\phi$, the curvature $F$ transforms by $\phi F \phi^{-1}$, for all $X\in \mathcal{X}$ there exists an infinitesimal gauge transformation $\phi_X$ such that 
\begin{align*}
   \mathcal{L}_XF&=[\phi_X,F].
\end{align*}
Also $i_XA$ defines an infinitesimal gauge transformation, hence $[\phi_X+i_XA,F]$ is an action of an infinitesimal gauge transformation on $F$.
Thus for all $X\in\mathcal{X}$ the deformations $i_XF$ arise from gauge transformations and hence do not descend to the moduli space.

Thus for finding the multiplicity of $\fg$ in the deformation space (modulo gauge transformations) of the canonical connection on $G/H$, we need to find the number of trivial sub-representations of $H$ in $\fg_2\otimes\fg_2$ apart from the one that corresponds to $F$. In all the cases we consider, the trivial $H$-representation occurs with multiplicity one in the subrepresentation  $\fg_2\otimes\h$ of  $\fg_2\otimes\fg_2$. The trivial representation coming from $\fg_2\otimes\h$ corresponds to the $H$-invariant element $F$.   We deal with the four normal homogeneous spaces one by one. The notation for the irreducible $H$-representations in all the cases is the same as used in \textsection\ref{evdsquare}.
\begin{enumerate}
\item[--]$\mathrm{Spin}(7)/\G2$

Since $H=\G2$, in this case $\fg_2$ is the irreducible adjoint representation. There is only one trivial $\fg_2$-subrepresentation of $\fg_2\otimes\fg_2$ which corresponds to $F$. Hence $\fg=\mathfrak{spin}(7)$ does not occur in the deformation space as proved in Theorem \ref{thm:deformspace}. 

\medskip
    \item[--]${\mathrm{SO}(5)}/ \mathrm{SO}(3)$
    
    In this case, as an $\mathfrak{so}(3)$ representation, $\fg_2$ decomposes into two irreducible $\mathfrak{so}(3)$-representations, the adjoint representation $S^2\C^2$, and the $11$-dimensional representation $S^{10}\C^2$. Thus as $\mathfrak{so}(3)$-representation
    \begin{align*}
        \fg_2\otimes\fg_2&= (S^2\C^2\otimes S^2\C^2)\oplus 2(S^2\C^2\otimes S^{10}\C^2)\oplus ( S^{10}\C^2\otimes  S^{10}\C^2).
    \end{align*}
    There are two trivial components occurring in the above decomposition from $S^2\C^2\otimes S^2\C^2$ and $S^{10}\C^2\otimes  S^{10}\C^2$ respectively but since the component coming from $S^2\C^2\otimes S^2\C^2$ corresponds to $F$, up to gauge transformations the deformation space of the canonical connection on $\mathrm{SO}(5)/\mathrm{SO}(3)$ contains only one copy of $\fg=\mathfrak{so}(5)$ as shown in Theorem \ref{thm:deformspace}.
    
\medskip
\item[--]$\mathrm{Sp}(2)\times\mathrm{Sp}(1)/\mathrm{Sp}(1)\times\mathrm{Sp}(1)$

As an $\mathfrak{sp}(1)\oplus\mathfrak{sp}(1)$-representation, 
\begin{align*}
    \fg_2&=S^2P\oplus S^2Q\oplus PS^3Q.
\end{align*}
The trivial $\mathfrak{sp}(1)\oplus\mathfrak{sp}(1)$ components of $\fg_2\otimes\fg_2$ coming from $S^2P\otimes S^2P$ and $S^2Q\otimes S^2Q$ correspond to $F$ and thus can be ignored. The only trivial component that corresponds to an infinitesimal deformation modulo gauge transformations comes from $PS^3Q\otimes PS^3Q$, hence again $\fg=\mathfrak{sp}(2)\oplus\mathfrak{sp}(1)$ appears with multiplicity 1 in the deformation space which is consistent with our findings in Theorem \ref{thm:deformspace}.

\medskip
\item[--]$\mathrm{SU}(3)\times\mathrm{SU}(2)/\mathrm{SU}(2)\times\mathrm{U}(1)$

The decomposition of $\fg_2$ as an $\mathfrak{su}(2)\oplus\mathfrak{u}(1)$-representation is given by
\begin{align*}
    \fg_2&=S^2W\oplus \C \oplus S^3WF(3)\oplus S^3WF(-3) \oplus F(6)\oplus F(-6).
\end{align*} The first two components in the above decomposition correspond to $\h$ hence the only trivial $\mathfrak{su}(2)\oplus\mathfrak{u}(1)$-subrepresentations of $\fg_2\otimes\fg_2$ that correspond to non-trivial deformations come from the spaces $S^2WF(3)\otimes S^2WF(-3)$ and $F(6)\otimes F(-6)$. Hence as proved in Theorem \ref{thm:deformspace} the space $\fg=\mathfrak{su}(3)\oplus\mathfrak{su}(2)$ occurs in the deformation space with multiplicity 2.

The only deformation spaces left to be considered in Table \ref{tabledeformationspace} are the $\mathrm{Sp}(2)$-representation $V_{(0,1)}^\R$ for the squashed $7$-sphere and 2 copies of the $\mathrm{SU}(3)$-representation  $\mathfrak{su}(3)$ on the Aloff--Wallach space $\mathrm{SU}(3)\times\mathrm{SU}(2)/\mathrm{SU}(2)\times \mathrm{U}(1)$. 
\end{enumerate}
On the squashed $7$-sphere the canonical connection splits into two connections with $\rm{Hol}=\rm{Sp}(1)_u$ and $\rm{Sp}(1)_d$ respectively. From \textsection \ref{evdsquare} the deformations only come from the $\rm{Sp}(1)_u$ part which is the pullback of the standard instanton on $S^4$. If we view $S^4$ as the symmetric homogeneous space $\frac{\rm{Sp}(2)\times\rm{Sp}(1)}{\rm{Sp}(1)_a\times\rm{Sp}(1)_b\times\rm{Sp}(1)_c}$ and denote by $P,Q,R\cong \C^2$ the irreducible representation of the three $\rm{Sp}(1)$ factors respectively we have the orthogonal decomposition \begin{align*}
    \mathfrak{sp}(2)\oplus\mathfrak{sp}(1)= \mathfrak{sp}(1)_a\oplus\mathfrak{sp}(1)_b\oplus\mathfrak{sp}(1)_c\oplus \mathfrak{n}.
\end{align*}
As an $\mathfrak{sp}(1)_a\oplus\mathfrak{sp}(1)_b\oplus\mathfrak{sp}(1)_c$-representation 
\begin{align*}
    \mathfrak{n}\cong PQ.
\end{align*}
The squashed sphere becomes a bundle over $S^4$ by reducing to the subgroup $\rm{Sp}(1)^2$ corresponding to the identification $Q=R$ so the factor $\rm{Sp}(1)_d$ acts diagonally. The complexified tangent space of $\frac{\rm{Sp}(2)\times\rm{Sp}(1)}{\rm{Sp}(1)_u\times\rm{Sp}(1)_d}$ is then  
   \begin{align*}
       m  \cong  S^2Q + PQ.
   \end{align*}
The standard instanton on $S^4$ is the unique $\rm{Sp}(2)$-invariant ASD connection on $S^4$ with charge 1. As a bundle over $S^4$, the Levi-Civita connection induces the standard instanton on $P$. It is also the homogeneous connection on the $\rm{Spin}(4)=\rm{Sp}(1)^2$ bundle over $S^4$ obtained by left-translating the subspace $\mathfrak{n}$ in
   $\mathfrak{sp}(2)\oplus\mathfrak{sp}(1) = 3\mathfrak{sp}(1) \oplus \mathfrak{n}$ 
by $\rm{Sp}(2)\times \rm{Sp}(1)$. Thus the horizontal distribution corresponding to the standard instanton is $\mathfrak{n}$. 

On the other hand the canonical connection on the squashed $7$-sphere is the characteristic homogeneous connection defined by the horizontal distribution $\m$ in the decomposition $\mathfrak{sp}(2)\oplus\mathfrak{sp}(1) = 2\mathfrak{sp}(1) \oplus \mathfrak{m}=2\mathfrak{sp}(1)\oplus(S^2Q\oplus \mathfrak{n})$. The canonical connection on squashed $7$-sphere reduces to $\rm{Sp}(1)^2$ and preserves the horizontal distribution $D$ defined by $\mathfrak{n}$ which is stable under both $\rm{Ad}(Sp(1)^3)$ and $\rm{Ad}(\rm{Sp}(1)^2)$.

If we consider the map $$p:S^7=\frac{\rm{Sp}(2)\times\rm{Sp}(1)}{\rm{Sp}(1)_u\times\rm{Sp}(1)_d}\to S^4=\frac{\rm{Sp}(2)\times\rm{Sp}(1)}{\rm{Sp}(1)_a\times\rm{Sp}(1)_b\times\rm{Sp}(1)_c}$$ 
then the connection induced on $D$ is the pullback of the homogeneous connection defined by $n$ on $T\left(S^4\right)$ via $p$.

Let $\mathcal{M}$ be the moduli space of charge-1 instantons on $S^4$ with structure group $\rm{SU}(2)$. Then, there is a diffeomorphism from $\mathcal{M}$ to $B^5 \subset \R^5$ which to an instanton associates its center. The standard instanton on $S^4$ is the charge-1 instanton that corresponds to the center of the ball, that is to $0 \in B^5$, and is the unique homogeneous charge-1 instanton. As the name suggests, the homogeneous charge-1 instanton is invariant with respect to the $\rm{Sp}(2)$-action. The pullback of the homogeneous charge-1 instanton to the squashed $S^7$ is a $\G2$-instanton (see \cite{gonball}, \cite{clarke}). As shown in \cite{AHS-ASDonS4} the moduli space of the standard instanton on $S^4$ can be identified as a topological space and as a differentiable manifold with $\R^+ \times \mathbb{H}$ (see \cite[sec 4.1]{Donaldson-Kronheimer}). As shown above the $\rm{Sp}(1)$ part of the canonical connection on the squashed $7$-sphere is the pullback of the standard instanton, hence the deformation space of the canonical connection on the squashed $7$-sphere must contain the deformation space of the standard ASD instanton on $S^4$ and thus be at least $5$-dimensional. From Table \ref{tabledeformationspace} we know that the moduli space of the deformations of the canonical connection on the squashed $7$-sphere is exactly $5$-dimensional and hence we get the following Corollary.

\begin{thm}\label{squashed-genuine}
The deformations of the canonical connection on the squashed $7$-sphere are lifts of the deformations of the standard ASD connection on $S^4$ and are thus integrable. 
\end{thm}

As of the deformation subspace isomorphic to $2\mathfrak{su}(3)$ of the canonical connection on $\mathrm{SU}(3)\times\mathrm{SU}(2)/\mathrm{SU}(2)\times \mathrm{U}(1)$ with structure group $\G2$, the author is unaware of any such explicit description. It would be interesting to see whether these deformations are genuine.

\phantomsection \addcontentsline{toc}{section}{References}
\bibliographystyle{amsalpha}

\bibliography{deformation}
\end{document}